\documentclass[a4paper,11pt,leqno]{article}

\usepackage[latin1]{inputenc}
\usepackage[T1]{fontenc} %only using latex
\usepackage[english]{babel}
\usepackage{verbatim}
\usepackage{amsfonts}
\usepackage{amsmath}
\usepackage{amsthm}
\usepackage{amssymb}
\usepackage{color}
\usepackage{graphicx}
\usepackage{bm}

\theoremstyle{definition}
\newtheorem{defin}{Definition}[section]

\newtheorem{rem}[defin]{Remark}

\theoremstyle{plane}
\newtheorem{thm}[defin]{Theorem}
\newtheorem{prop}[defin]{Proposition}
\newtheorem{coroll}[defin]{Corollary}
\newtheorem{lemma}[defin]{Lemma}

\newcommand{\mbb}{\mathbb}

\newcommand{\mc}{\mathcal}
\newcommand{\mf}{\mathfrak}
\newcommand{\veps}{\varepsilon}
\newcommand{\what}{\widehat}
\newcommand{\wtilde}{\widetilde}
\newcommand{\vphi}{\varphi}

\newcommand{\ra}{\rightarrow}

\newcommand{\hra}{\hookrightarrow}

\newcommand{\g}{\gamma}
\newcommand{\lan}{\langle}
\newcommand{\ran}{\rangle}

\newcommand{\R}{\mathbb{R}}

\newcommand{\N}{\mathbb{N}}
\newcommand{\Z}{\mathbb{Z}}
\newcommand{\T}{\mathbb{T}}

\renewcommand{\div}{{\rm div}\,}

\newcommand{\Id}{{\rm Id}\,}

\allowdisplaybreaks

\def\d{\partial}
\def\div{{\rm div}\,}

\title{\textsc{\Large{\textbf{Geophysical flows and \\ the effects of a strong surface tension}}}}

\author{\normalsize\textsl{Francesco Fanelli} \vspace{.3cm} \\
{\small \textit{Institut Camille Jordan - UMR CNRS 5208}} \vspace{.1cm} \\
{\small \textsc{Universit\'e de Lyon, Universit\'e Claude Bernard Lyon 1}} \vspace{.1cm} \\
{\scriptsize {B\^atiment Braconnier}} \\
{\scriptsize {43, Boulevard du 11 novembre 1918}} \\
{\scriptsize {F-69622 Villeurbanne cedex -- FRANCE}} \vspace{.2cm} \\
{\small \ttfamily{fanelli@math.univ-lyon1.fr}} }
\vspace{.2cm}
\date\today

\begin{document}

\maketitle

%%%%%%%%%%%%%%%%%%%%%%%%%%%%%%%%%%%%%%%%%%%%%%%%%%%%%%%%%%%%%%%%%%%%%%%%%%%%%%%%%%%%%%%
\subsubsection*{Abstract}
{\small In the present note we review some recent results for a class of singular perturbation problems for a Navier-Stokes-Korteweg system with Coriolis force.
More precisely, we study the asymptotic behaviour of solutions when taking incompressible and fast rotation limits simultaneously, in a constant capillarity regime.

Our main purpose here is to explain in detail the description of the phenomena we want to capture, and the mathematical derivation of the system of equations.
Hence, a huge part of this work is devoted to physical considerations and mathematical modeling.}

\paragraph*{2010 Mathematics Subject Classification:}{\small 35B25 % PDEs / Qualitative properties / Singular perturbations
(primary);
35Q35, % PDEs / Eq. of math. phys. / PDEs in connection with fluid mechanics
35B40, % PDEs / Qualitative properties / Asymptotic behavior of solutions
35Q86, % PDEs / Equations of mathematical physics and other areas of application / PDEs in connection with geophysics
76D45, % Fluid Mechanics / Incompressible viscous fluids / Capillarity (surface tension)
76M45, % Fluid Mechanics / Basic methods in Fluid Mechanics / Asymptotic methods, singular perturbations
76U05, % Fluid Mechanics / Rotating fluids / Rotating fluids
82C24 % Statistical mechanics, structure of matter / Time-dependent statistical mechanics (dynamic and nonequilibrium) / Interface problems; diffusion-limited aggregation
(secondary).}

\paragraph*{Keywords:}{\small Capillarity; Earth rotation; Navier-Stokes-Korteweg system; singular perturbation problem;  low Mach, Rossby and Weber numbers;
variable rotation axis; Zygmund regularity.}

%%%%%%%%%%%%%%%%%%%%%%%%%%%%%%%%%%%%%%%%%%%%%%%%%%

\section{Introduction} \label{s:intro}

In the present note, we consider a mathematical model for compressible viscous fluids whose dynamics is mainly influenced by two
effects: strong surface tension and fast rotation of the ambient reference system (the Earth for us, but it is also the case of e.g. other planets or stars).

We intend to review some recent results for a class of singular perturbation problems for our system. More precisely, we study the asymptotic
behaviour of solutions when taking incompressible and fast rotation limits simultaneously, in a constant capillarity regime.
The analysis we present has mostly been performed in \cite{F_2015}-\cite{F_2016}: we refer to these works for more details and further references.
The main purpose of this work is rather to explain in detail the physical description of the phenomena we want to capture, and their translation at the mathematical level.

Hence, in a first moment (see Section \ref{s:model}) we derive the mathematical model we want to study, namely a Navier-Stokes-Korteweg system with degenerate viscosity coefficient
and Coriolis term.

It is a compressible Navier-Stokes type system, with an additional term due to capillarity and which depends on higher order space derivatives of the density. 
The form of the so-called Korteweg stress tensor we consider here was firsty introduced by Dunn and Serrin in \cite{D-S} (see also Subsection \ref{ss:i_cap} below
for further details); it is nothing but a choice among all the possible ones: different forms lead to different models, which are relevant in e.g. quantum hydrodynamics or shallow water theory.

The reason of considering a density-dependent viscosity, which vanishes in vacuum regions, is twofold.
The first motivation comes from modeling purposes: indeed, in general the viscosity of a fluid depends both on density and temperature, even in the Newtonian case (see e.g. \cite{Chap-Cow}).  
The second one is purely mathematical: although this choice complicates the analysis, since no information can be deduced
on the velocity field $u$ when the density $\rho$ vanishes, it enables us to exploit a fundamental underlying structure of the equations, the so-called \emph{BD entropy structure},
which in turn allows to gain additional regularity for $\rho$ from the capillarity term. We refer to Subsection \ref{ss:cap-effects} below for more details.

Finally, the additional Coriolis term comes from the fact that we are interested in dynamics of geophysical fluids, for which effects due to Earth rotation cannot be neglected.
At the mathematical level, its presence in the equations of motion does not involve any problem as far as one is interested in classical energy estimates, which is usually the case. On the contrary,
here it will cause some troubles in the analysis, due to its (apparent) incompatibility with the BD structure. We will explain better this point in Section \ref{s:math_prop}, where
we turn our attention to the mathematical properties of our system.

In the last part of the paper (see Sections \ref{s:asympt} and \ref{s:proof}), we specialize on the physically relevant case of variable rotation axis,
namely depending both on latitude and longitude.
As said at the beginning, we show a result on asymptotic behaviour of weak solutions, performing together the fast rotation and incompressible limits in the regime of constant capillarity.
Besides, we refine the result of \cite{F_2016} concerning minimal regularity assumptions on the fluctuation of the axis, introducing Zygmund-type conditions on the gradient of
the variation function. We refer to the above mentioned works \cite{F_2015}-\cite{F_2016} for further results, for an overview of related studies and additional references.

Let us complete this brief introduction by a detailed description of the physical effects we want to capture here, namely Earth rotation and surface tension.

\subsection{Earth rotation and the Coriolis force} \label{ss:i_rot}

Coriolis force is an inertial force, which shows up because of the choice of a rotating reference frame for describing a motion.
It causes a travelling object to curve its trajectory: actually this deflection is just apparent, and it is ``seen'' only by an observer in the rotating frame.
Coriolis force depends on the rotation speed of the reference frame and on the speed of the object through their cross product;
in particular, for motions on the Earth it depends on the projection of the rotation axis on the local vertical (i.e. the latitude).

Obviously, Earth is a rotating frame, and we have to keep into account this kind of effects when describing an event: indeed Coriolis force can be experienced in several ways in everyday life.
%We do not detail here all these phenomena, limiting ourselves to mention that rotation
Its consequences are especially evident in the motion of geophysical fluids, due to the large scales
involved (see also the discussion below). In fact, rotation is one of the two main ingredients which distinguish geophysical flows, the other
one being stratification. 
%Before presenting a review of some physical phenomena linked with geophysical fluids, let us remark that, quite surprisingly, effects due to centrifugal forces can be completely
%neglected at this level. As a matter of fact, the outward pull caused by the centrifugal force alters the spherical equilibrium due to gravity: as a consequence,
%the Eart assumes a slightly flattened shape. Nonetheless, the gravity by far exceeds the centrifugal force, so that the resulting force points inward, aligned with the local vertical.
%As a consequence, a particle at rest on the surface remains at rest unless it is subject to additional forces: the flattening (which is very slight, however) is precisely that necessary
%to neutralize the centrifugal force. Therefore, as it is common use (see e.g. Chapter 2 of \cite{CR}), in the following we call ``gravitational force'' the sum of the true gravitational
%force and the centrifugal force, aligned with the local vertical.
Quite surprisingly, centrifugal forces can be completely neglected at this level, since it is by far exceeded by gravity, which essentially rubs out its effects (up to small corrections
in the geometry of the planet). We refer to \cite{CR} for further details.

We can justify the relevance of Coriolis force in geophysical fluid dynamics by arguments based on orders of magnitude, and in fact one can speak about \emph{fast rotation}, since the
time of the physical process is much larger than the one taken by the Earth to make a revolution. This will be our approach throughout this work;
we will give more details in Subsection \ref{ss:singular}. For the time being, let us make a brief overview of the main physical phenomena determined by the Earth rotation.

The first remarkable effect is that fast rotation induces a sort of rigidity on the fluid motion, as discovered first by the physicist Taylor. More precisely,
it creates an asymmetry between vertical (i.e. parallel to the rotation axis) and horizontal
(i.e. in the plane orthogonal to the rotation axis) components of the velocity field, and the fluid particles tend to move on vertical columns.
Of course, such a vertical rigidity (known as \emph{Taylor-Proudman theorem}) is not perfectly realized in oceanic and athmospheric circulation, mainly because the rotation is not fast
enough and the density not uniform enough to hide other ongoing physical processes.

These discrepancies, as well as other deviations from idealized configurations, generate and propagate along \emph{waves}, which can be of various nature: one speaks about Kelvin, Poincar\'e,
Rossby and topographic waves. We refer to \cite{CR} and \cite{Ped} for a more indeep discussion. As far as we are concerned, we will deal with Poincar\'e and Rossby waves:
the former are due to non-vanishing vertical components of the velocity field of the fluid, while the latter are due to ``variations'' of the rotation axis together with latitude
(we will be more precise later on, see in particular Subsection \ref{ss:rot-effects}).
Actually, this terminology does not completely fit the cases we encounter here: indeed, we will consider slightly compressible fluids, which produce then acoustic waves
(as it is well-known, see e.g. \cite{Light}) due to the compressible part of the velocity field. Hence, our analysis will rather reveal the interactions between these
different classes of waves (acoustic and geophysical). Let us remark that, although having variable density, our fluids are not exactly stratified in the common sense
of geophysics, since we neglect here the gravitational force for the sake of simplicity: so, in particular we will not see vertical motions.

The vertical rigidity sancioned by the Taylor-Proudman theorem is somehow incompatible also with the physical condition which prescribes the fluid to be at rest at
the boundary (i.e. Dirichlet boundary conditions). Such an apparent incompatibility is actually explained by the formation of boundary layers, called \emph{Ekman layers}, which
are due to vertical friction. Vertical friction effects are very small if compared with rotation (this assertion can be justified again by scaling arguments),
and they can be omitted in a first approximation, but they cause the motion, roughly speaking, to present two distinct behaviours: in the interior friction is usually negligible,
while near the boundary and across a small distance (which is in fact the boundary layer) it acts to slow down the interior velocity up to zero.
Physical considerations show that this process is not purely horizontal, but it creates a vertical velocity, the so-called \emph{Ekman pumping},
which occurs throughout the depth of the fluid. This vertical component of the velocity produces a global circulation effect from the interior to the Ekman
layer and conversely, which is reponsible for the dissipation of a huge amount of kinetic energy.
As a final remark, we point out that actually two factors (among others) account for substantial corrections to this quite idealized description: turbulence
(geophysical flows have very large Reynolds number) and stratification. We do not detail this point here, for which we rather refer to Chapter 4 of \cite{Ped} and
Chapter 5 of \cite{CR}.

\medbreak
Let us now consider fluids which are influenced by the presence of a large surface tension.

\subsection{Capillarity: some physical insights} \label{ss:i_cap}

Capillarity is a physical effect which is intimately linked with surface tension of a fluid, and indeed these two words (i.e. capillarity and surface tension) are often used as synonims,
at least by mathematicians. This will be also our point of view throughout all this paper.

Capillary phenomena are connected with large variations of the density function in small regions: for instance whenever different fluids touch each other, or when matter presents different states.
In everyday life, capillarity can be commonly experienced observing fluids in thin tubes: this sentence includes a lot of simple events, like drinks filling a straw or
being absorbed by a paper (or other material) towel. Likewise, surface tension is responsible for droplets formation and breakup. Capillarity is suspected to play a fundamental role
even in phenomena linked with superfluids (like e.g. liquid helium or ultracold atomic gases); superfluids appear in astrophysics, high-energy physics and quantum gravity theories.

In a more rigorous way, from the physical viewpoint surface tension plays a fundamental role in describing phenomena which occur at the interface between
two (miscible or immiscible) fluids, see e.g. \cite{A-McF-W}.
In classical fluid mechanics approach, the interface is assumed to be of zero thickness and it is modeled then as an evolving free boundary. This description, although being successfull
in many contexts, turns out to fail in various situations: in particular, it breaks down whenever the physical process takes place at lenght scales which are comparable to
the thickness of the interface. In this case, \emph{diffuse-interface} models provide an alternative description, which is able to overtake problems of that kind. The leading idea of
these models is that quantities which were localized in the interfacial surface before, are now distributed throughout a region of possibly very small but non-zero thickness.
It is exactly in this interfacial region that density experiences steep (but still smooth) changes of value.

Diffuse-interface models are used also in the modern theory of capillarity. Previously (at around the beginning of the XIX century), capillary phenomena were explained as
the result of the interaction between short-ranged attraction forces and ``repulsive forces'' between molecules; these latter were in fact interpreted by Laplace as
a byproduct of an internal intermolecular pressure. Laplace's description evolved through Maxwell and Lord Rayleigh up to van der Waals: although being based on a static
view of matter, this theory is in fact able to explain many physical phenomena in a quite accurate way. We refer to Chapter 1 of \cite{Row-Wid} for many additional details.

From both points of view (the one of short-ranged interactions and the other one of diffuse-interfaces), the relevant feature is that density undergoes large variations in regions of small thickness.
This was already remarked by Korteweg, who proposed a new constitutive law for the stress tensor of the fluid under consideration:
an ``elastic'' component had to be added to the classical form of Cauchy and Poisson, which had to depend on the density and its derivatives up to second order
(which corresponds, in modern terminology, to fluids of grade $2$). Adopting a rational mechanics approach, in \cite{D-S} Dunn and Serrin showed the incompatibility of
the form proposed by Korteweg with the classical principle of thermodynamics, and they proposed a new version of the capillarity stress tensor, which can actually be generalized to
model fluids of any grade $N$.
Finally, it is the Dunn--Serrin form of the Korteweg stress tensor which is used in mathematical studies: see e.g. \cite{B-D-L}, \cite{B-D_2003} and \cite{BG-D-Desc_2007} (this last reference
concerns the case of inviscid fluids); we also refer to \cite{BG-D-Desc-J} for the study of traveling waves and stability issues linked with inviscid capillary flows, and
to \cite{B-C-N-V} for an interesting reformulation of the Korteweg tensor and related properties.

In Subsection \ref{ss:korteweg} we will derive the Dunn--Serrin form of the Korteweg tensor by use of variational arguments, in the special case of fluids of grade $2$.
We refer to the introductions of \cite{F_2015}-\cite{F_2016} for further references and mathematical results concerning related capillary models.

\subsection{Relevance of mixing rotation and capillarity}

We conclude the introduction by justifying our study, where we look at both rotation and surface tension effects in the motion.

First of all, we remark that the diffuse-interface approach is relevant also for single component fluids with variable density. Therefore, it is pertinent
to consider a stress tensor of Korteweg type in the context of geophysical flows, assuming that internal forces generate a significant part of the internal energy
of the fluid.

Notice also that interface mechanisms play a relevant role in propagation of internal waves; keeping them into account is hence fundamental in describing athmosphere and ocean dynamics
(see Chapter 10 of \cite{CR} about this point). Adding a capillarity term in the equations of motion can be seen as a (maybe rough, but still
interesting, from our point of view) attempt of capturing these phenomena.

Finally, as we will see in detail in Subsection \ref{ss:singular}, our equations present a strong similarity with a $2$-D shallow water system, used to describe the so-called
\emph{betaplane model} for geophysical dynamics close to the equatorial zone (see e.g. \cite{B-D-Met}, \cite{G} and \cite{G-SR_Mem}). Our study can be put in relation with those works:
here we consider a $3$-D domain and a degenerate viscosity coefficient.

%%%%%%%%%%%%%%%%%%%%%%%%%%%%%%%%%%%%%%%%%%%%%%%%%%%%%%%%%%%%%%%%%%%%%%%%%%%%%%%%%%%%%%%%%%%%%%%%%%%%%%%%%%%%%%%%%%%%%%%%%%%%%%%%%%%%%%%
\subsubsection*{Acknowledgements}
The study presented in this note has been widely enriched by interesting discussions with \textit{Sylvie Benzoni-Gavage}, \textit{Didier Bresch} and \textit{Isabelle
Gallagher}. The author wants to express all his gratitude to them.

This work was partially supported by the LABEX MILYON (ANR-10-LABX-0070) of Universit\'e de Lyon, within the program ``Investissement d'Avenir''
(ANR-11-IDEX-0007) operated by the French National Research Agency (ANR).

%%%%%%%%%%%%%%%%%%%%%%%%%%%%%%%%%%%%%%%%%%%%%%%%%%

\section{A model for rotating capillary fluids} \label{s:model}

In this section we introduce our model, namely a Navier-Stokes-Korteweg system with Coriolis force.
It presents a term which keeps into account capillarity effects, the so-called Korteweg stress tensor, and another one which is due to the rotation of the Earth.

We start by reviewing some physical background and explaining how these effects can be translated at the mathematical level.
This will lead us to write down the equations of motion in Subsection \ref{ss:motion}. In the final Subsection \ref{ss:singular},
we will adopt the singular perturbation analysis point of view, presenting the rescaled system.

%%%%%%%%%%
\subsection{On the Korteweg stress tensor} \label{ss:korteweg}

Let us consider first surface tension effects on fluid motion.
%This part is essentially taken from the lecture notes \cite{BG_Levico}: we refer to them
%for further details and references.

As pointed out in the introduction, historically capillarity phenomena were described as a consequence of short-ranged (attractive and repulsive) forces which are produced between
molecules. On the other hand, more recent theories are based on a diffuse-interface approach, where however the interfacial region can have a very small thickness.

In both contexts, it is then quite natural to assume that the classical stress tensor of the fluid under consideration has to be modified, in order to include
terms which depend on the gradient and higher order derivatives of the density. Here, however, we want to take a different point of view: namely, we are going to derive its form (one of the
possible choices) by variational principles, postulating just the specific expression of the free energy. We refer to \cite{A-McF-W} and \cite{BG_Levico} for more details.

For simplicity, let us consider a non-uniform single-component fluid at equilibrium, and let us denote by $\rho$ its density and by $\theta$ its temperature.
We suppose that the \emph{Helmholtz free energy} functional takes the form
\begin{equation} \label{eq:Helmholtz}
\mc{E}[\rho,\theta]\,:=\,\int_\Omega\left(\rho\,e_0(\rho,\theta)\,+\,\frac{1}{2}\,k(\rho,\theta)\,|\nabla\rho|^2\right)\,dx
\end{equation}
in some domain $\Omega\subset\R^3$, where $dx$ is the volume measure.
Of course, for a fluid not necessarily at rest, the contribution coming from the kinetic energy
$$
\mc{E}_k[\rho,u]\,:=\,\frac{1}{2}\,\int_\Omega\rho\,|u|^2\,dx
$$
has to be added to $\mc{E}$, where $u\in\R^3$ represents the velocity field of the fluid.

In formula \eqref{eq:Helmholtz}, the function $e_0$ is the bulk free energy density per unit mass, and the first term in the integral, i.e.
$$
\mc{E}_0[\rho,\theta]\,:=\,\int_\Omega\mc L_0[\rho,\theta]\,dx\,,\qquad\qquad\mbox{ with }\qquad \mc L_0[\rho,\theta]\,:=\,\rho\,e_0(\rho,\theta)\,,
$$
represents the ``classical'' energy of a fluid at rest. The second term in the integral takes into account the presence of capillarity effects in the fluid,
in which large density variations are no more negligible and contribute to free energy excess of the interfacial region. The positive function $k(\rho,\theta)$ is the internal
energy coefficient, also called capillarity coefficient; it is possible to give it a more precise meaning in the context of statistical mechanics,
but this point goes beyond the scopes of this discussion, and we will not say more about it here.

From now on, we restrict our attention to the case of \emph{isothermal fluids}. So let us set (with no loss of generality) $\theta\equiv1$ and forget about it in the notations.

The equilibrium conditions are obtained by minimizing the function $\mc{E}$ under the constraint of constant mass $\mc M\,:=\,\int_\Omega\rho\,dx$.
Therefore, defined the Lagrangian function
\begin{equation} \label{eq:E-L}
\mc{L}[\rho,\nabla\rho]\,:=\,\rho\,e_0(\rho)\,+\,\frac{1}{2}\,k(\rho)\,|\nabla\rho|^2\,-\,\lambda\,\rho
\end{equation}
(where $\lambda$ is a Lagrange multiplier), by Hamilton's principle one finds the Euler-Lagrange equation
$$
\frac{1}{2}\,k'(\rho)\,|\nabla\rho|^2\,+\,k(\rho)\,\Delta\rho\,-\,\d_\rho\bigl(\rho\,e_0(\rho)\bigr)\,+\,\lambda\,=\,0\,.
$$

On the other hand, both $\mc L_1[\rho,\nabla\rho]\,:=\,\mc{L}_0[\rho]\,+\,k(\rho)\,|\nabla\rho|^2/2$ and the mass constraint are independent of spatial coordinates,
and therefore they are invariant by the action of the vector fields $\d_j$, for all $1\leq j\leq 3$. Then, by application of Noether's theorem we get the relation
$\div J\,=\,0$, where we have defined the $2$-tensor
$$
J\,:=\,\mc{L}\,\Id\,-\,\nabla\rho\,\otimes\,\frac{\d\mc L}{\d(\nabla\rho)}\,.
$$
Replacing $\mc L$ by its definition and using \eqref{eq:E-L} to find $\lambda$, we arrive at the expression
\begin{equation} \label{eq:static-stress}
J\,=\,-\,k(\rho)\,\nabla\rho\otimes\nabla\rho\,+\,\left(-\,\rho^2\,e_0'(\rho)\,+\,\frac{1}{2}\bigl(k(\rho)\,+\,\rho\,k'(\rho)\bigr)\,|\nabla\rho|^2\,+\,\rho\,k(\rho)\,\Delta\rho\right)\Id\,.
\end{equation}

We remark that $\Pi\,:=\,\rho^2\,e_0'(\rho)$ is the standard thermodynamic pressure. It is classically given by $\Pi\,=\,\rho\,\d_\rho\mc L_0\,-\,\mc L_0$; this formula allows us to
define a \emph{generalized pressure} for capillary fluids as
$$
\Pi_g\,:=\,\rho\,\d_\rho\mc L\,-\,\mc L\,=\,\Pi\,+\,\frac{1}{2}\,\bigl(\rho\,k'(\rho)\,-\,k(\rho)\bigr)\,|\nabla\rho|^2\,.
$$
Then, with this definition, $J$ assumes the form
$$ %\begin{equation}
J[\rho,\nabla\rho]\,=\,-\,k(\rho)\,\nabla\rho\otimes\nabla\rho\,+\,\Bigl(-\,\Pi_g\,+\,k(\rho)\,\div\bigl(\rho\,\nabla\rho\bigr)\Bigr)\,\Id\,.
$$ %\end{equation}
Let us mention that $J$, given by formula \eqref{eq:static-stress} and derived here in the static case, actually represents
the reversible part of the stress tensor for a capillary fluid which is not at equilibrium (of course, viscosity is missing in the present discussion).
We refer to \cite{A-McF-W} for additional details about this point, which will be used in Subsection \ref{ss:motion} to write the equations of motion.

Finally, for later use, we define the \emph{Korteweg stress tensor} as
\begin{equation} \label{eq:K-stress}
K[\rho,\nabla\rho]\,:=\,-\,k(\rho)\,\nabla\rho\otimes\nabla\rho\,+\,\frac{1}{2}\Bigl(\bigl(k(\rho)\,+\,\rho\,k'(\rho)\bigr)\,|\nabla\rho|^2\,+\,\rho\,k(\rho)\,\Delta\rho\Bigr)\,\Id\,.
\end{equation}

\begin{rem} \label{r:k}
Of course, different forms of $K[\rho,\nabla\rho]$ are possible, depending on the form of the interfacial energy $\mc E\,-\,\mc E_0$. This leads to several models, which
are relevant in various contexts, like e.g. in quantum hydrodynamics or in shallow water theory.
%We will give a quick overview of them in Paragraph \ref{sss:related}, after presenting the dynamical model we want to study.
We refer to \cite{BG_Levico} and to the introductions of \cite{F_2015}-\cite{F_2016} for more details about this point.
\end{rem}

\begin{rem} \label{r:phase-b}
The previous form of the total energy \eqref{eq:Helmholtz} has special importance when $\mc L_0$ is \emph{not} a convex function of $\rho$, like
e.g. for van der Waals equations of state. In these cases, $\mc E$ may admit minimizers which are not necessarily constant density profiles and for which, in particular,
two different phases are permitted to coexist. This allows one to capture phase boundaries and their dynamical evolution (see also Section 1 of \cite{BG_Levico} for
a more indeep discussion).
\end{rem}

%%%%%%%5%%
\subsection{Adding the rotation of the Earth} \label{ss:rotation}

We now turn our attention to Earth rotation and its mathematical description.

Given two frames $\mc F$ and $\mc F'$ in non-intertial relative motion, let us call $x(t)$ the position of a moving body with respect to the first origin
(say, the ``absolute'' position) and $x'(t)$ the position with respect to the second one (say, the position in the ``non-inertial'' frame). Denote by $o(t)$ the vector position of the origin
of $\mc F'$ with respect to the origin of $\mc F$, % (obviously, this choice is arbitrary, since what really matter is the relative motion between the two frames). 
and suppose also that $\mc F'$ is rotating with respect to $\mc F$, so that
$$
x(t)\,=\,o(t)\,+\,R(t)\,x'(t)\,,
$$
for some rotation matrix $R(t)$. We set $v(t)$ and $v'(t)$ to be the velocities in $\mc F$ and $\mc F'$ respectively, and $a(t)$ and
$a'(t)$ to be the accelerations: it is a standard matter (see e.g. \cite{Arnold} or \cite{Tartar}) to derive, from principles of Newtonian mechanics, the formula
$$
a(t)\,=\,a_o(t)\,+\,R(t)\Bigl(a'(t)\,+\,\xi(t)\times v'(t)\,+\,\dot{\xi}(t)\times x'(t)\Bigr)\,+\,R(t)\,\xi(t)\times\Bigl(v'(t)\,+\,\xi(t)\times x'(t)\Bigr)\,,
$$
where we have denoted by $a_o$ the relative acceleration of $\mc F'$ with respect to $\mc F$, by $\xi(t)$ the vector associated to the
skew-symmetric matrix $\,^tR(t)\,\dot{R}(t)$ (since we are in $\R^3$), and by a dot the time derivative.

The term $\xi\times(\xi\times v')$is related, together with the relative acceleration $a_0$, to the centrifugal force. It can be easily seen that it derives from a potential,
which slightly changes the gravitational potential creating the so-called geopotential. As we have said in Subsection \ref{ss:i_rot}, this term can be safely neglected
in a first approximation. The term $\dot{\xi}\times x'$ can be ingnored as well at this stage, since we suppose no variations in time of the Earth rotation. 
Finally, the term $2\,\xi\times v'$ is the \emph{Coriolis acceleration}; notice that, being orthogonal to the velocity, it makes no work and so
it is not seen in the energy balance (also, energy being an invariant, its form does not depend on the choice of the reference frame).

To sum up, after the previous approximations, the absolute acceleration reduces to the formula
\begin{equation} \label{eq:accel}
a(t)\,\simeq\,R(t)\Bigl(a'(t)\,+\,2\,\xi(t)\times v'(t)\Bigr)\,.
\end{equation}

Let us now place on the surface of the Earth, which we approximate to a perfect sphere (we neglect again effects due to centrifugal force). We also assume that it rotates
around its north pole-south pole axis: at any given latitude $\vphi$, this direction departs from the local vertical of the angle $\pi/2-\vphi$, and the Coriolis
force consequently varies. More precisely, let us fix a local chart on the Earth surface, such that the $x^1$-axis describes the longitude, oriented eastward, the $x^2$-axis describes the latitude,
oriented northward, and the $x^3$-axis describes the altitude, oriented upward; let us denote by $(e^1,e^2,e^3)$ the corresponding orthonormal basis.
Then, the Earth's rotation vector can be written in the following way:
$$
\xi\,=\,|\xi|\,\bigl(\cos\vphi\,e^2\,+\,\sin\vphi\,e^3\bigr)\,. %\,=\,|\xi|\,(0,\cos\vphi,\sin\vphi)\,.
$$

We are now going to make several assumptions.
First of all, for the sake of simplicity we neglect the curvature of the Earth, so that spherical coordinates are treated as cartesian coordinates. This approximation
is justified if we restrict our attention on regions of our planet which are not too extended with respect to the radius of the Earth (see also \cite{SR}).
In addition, we suppose that the rotation axis is parallel to the $x^3$-axis (i.e. we take $\vphi=\pi/2$), which is quite reasonable if we place far enough from the equatorial zone.
Nonetheless, in order to still keep track of the variations due to the latitude, we postulate the presence of a suitable smooth function $\mf{c}$ of the ``horizontal'' variables
$x^h\,=\,(x^1,x^2)$. % (i.e. the longitude and latitude).
Therefore, under the previous assumptions we can write
$$
\xi\,=\,\mf{c}(x^h)\,e^3\,. %\,=\,\mf{c}(x^h)\,\bigl(0,0,1\bigr)\,.
$$

In the light of the previous discussion, the equations of motion will be perturbed by adding one term (see the next subsection), i.e. the Coriolis operator
\begin{equation} \label{eq:coriolis}
\mf{C}(\rho,u)\,:=\,\mf{c}(x^h)\,e^3\,\times\,\rho\,u\,,
\end{equation}
where, as in the previous subsection, $\rho$ is the density of the fluid
and $u$ its velocity field. We remark that, since Coriolis force makes no work, this term does not contribute to the energy balance, neither to the pontential part $\mc E$ nor to
the kinetic part $\mc E_k$ (both defined in the previous subsection).

\begin{rem} \label{r:coriolis}
Notice that, in view of our assumptions, we have arrived at a quite simple form of the Coriolis operator. This form is nonetheless widely accepted in the mathematical community,
and it already enables one to capture different physical effects linked with rotation (see e.g. book \cite{C-D-G-G} and the references therein, and works \cite{G-SR_2006} and \cite{G-SR_Mem},
among many others).
Treating the most general form of the Coriolis operator goes beyond the scopes of our presentation.
\end{rem}

Precise hypotheses on the variation function $\mf c$ will be described in Subsections \ref{ss:hypotheses} and \ref{ss:var-result} below.
There, we will introduce also further simplifications to the model, in order to tackle the mathematical problem. 

For the time being, let us introduce the relevant system of equations.

%%%%%%%%%%
\subsection{Writing down the equations of motion} \label{ss:motion}

In the present subsection we finally present the system of PDEs which describes the physical phenomena we detailed above.
Some preliminary simplifications are in order, to deal with an accessible mathematical problem.

\subsubsection{Some assumptions} \label{sss:mot_assump}
As already said, we omit here effects due to centrifugal and gravitational forces. We will also introduce (see Subsection \ref{ss:var-result}) some restrictions on the variations of
the rotation axis, namely on the function $\mf c$ of formula \eqref{eq:coriolis}.

Moreover, we do not consider other relevant physical quantities: e.g. temperature variations, salinity (in the case of oceans), wind force (on the surface of the ocean).
Therefore, our fluid will be described by its density $\rho\geq0$ and its velocity field $u\in\R^3$. Moreover, we will restrict our attention to the case of
\emph{Newtonian fluids}, for which the viscous stress tensor is linearly dependent on the gradient of the velocity field: we suppose it to be given by the relation
$$
S[\rho,u]\,=\,\nu_1(\rho)\,Du\,+\,\nu_2(\rho)\,\div u\,\Id\,,\qquad\qquad\mbox{ with }\qquad Du\,=\,\nabla u\,+\,^t\nabla u\,,
$$
for suitable functions $\nu_1$ and $\nu_2$ of the density only. As already pointed out, considering density--dependent viscous coefficients is very important
at the level of mathematical modeling (but again, we are missing effects linked with temperature variations).
Here and until the end of the paper, we make the simple choice $\nu_1(\rho)\,=\,\nu\,\rho$ and $\nu_2(\rho)\equiv0$, for some constant $\nu>0$.

Finally, let us immediately fix the spacial domain: neglecting the curvature of the Earth surface for simplicity (as remarked in the previous subsection),
we consider evolutions on the infinite slab
$$
\Omega\,:=\,\R^2\,\times\;]0,1[
$$
(but the main result, i.e. Theorem \ref{t:sing_var} below, works also on $\T^2\,\times]0,1[\;$). Hence, we are going to consider a very simple geometry of the domain, which in particular
introduce some rigidity at the boundary: we ignore variations at the surface of the Earth for atmospheric circulation, and for ocean in its topografy and at the free surface
(what is called the \emph{rigid lid approximation}).

Let us reveal in advance that, for the sake of simplicity, we want to avoid boundary layers formation in the present study, and rather focus on the other physical effects
linked with fast rotation (recall the discussion in Subsection \ref{ss:i_rot}). Therefore, in dealing with the singular perturbation problem, we will take \emph{complete slip}
boundary conditions for $\Omega$: we refer to Subsection \ref{ss:hypotheses} for the precise assumptions.

\subsubsection{The system of equations} \label{sss:eq}
Thanks to these hypotheses, we can now write down the equations describing the dynamics of our geophysical flow. We do not present here the details of the precise derivation, which can
be found in many textbooks: see e.g. \cite{Feir}, \cite{F-N} and \cite{Gill}.

The first relation is deduced from the principle of mass conservation, namely $\mc M\,\equiv\,cst$, where $\mc M$ has been introduced in Subsection \ref{ss:korteweg}.
Differentiating it with respect to time and using the divergence theorem, one gets the equation
\begin{equation} \label{eq:mass}
\d_t\rho\,+\,\div\bigl(\rho\,u\bigr)\,=\,0\,.
\end{equation}

Gathering conservation of momentum is just a matter of applying Newton's second law $F\,=\,m\,a$, with the acceleration $a$ which is given
by \eqref{eq:accel} and with the mass $m$ replaced by the density (mass per unit of volume) $\rho$. Passing in Eulerian coordinates, so that $d/dt\,=\,\d_t\,+\,u\cdot\nabla$,
and denoting by $\Psi$ the stress tensor\footnote{We recall that the \emph{stress tensor} yields the force per unit of surface which the part of the fluid in contact with an ideal surface element
imposes on the other part of the fluid on the other side of the same surface element.}, we find
$$
\rho\,\bigl(\d_tu\,+\,u\cdot\nabla u\bigr)\,+\,\mf{C}(\rho,u)\,=\,\div\Psi\,,
$$
where we have also made use of \eqref{eq:coriolis}. By previous assumptions, in our case $\Psi$ is given by a more general form of the Stoke law, namely
$$ %\begin{eqnarray*}
\Psi\,=\,S[\rho,u]\,+\,J[\rho,\nabla\rho]\,=\,S[\rho,u]\,-\,\Pi(\rho)\,\Id\,+\,K[\rho,\nabla\rho]\,, 
$$ %\end{eqnarray*}
where $J$ and $K$ have been defined in \eqref{eq:static-stress} and \eqref{eq:K-stress} respectively. Recall that $\Pi$ is the thermodynamical pressure of the fluid:
we will specify its precise form in Subsection \ref{ss:hypotheses} below.

At this point, we make the choice $k(\rho)\equiv k$ constant in \eqref{eq:K-stress}, which we can take equal to $1$ with no loss of generality.
It is not just a simplification: at the mathematical level, this precise form of the Korteweg stress tensor is exactly the one which combines well with the previous fixed values of the
viscosity coefficients $\nu_1$ and $\nu_2$, in order to exploit the BD structure of our system (recall the discussion in the Introduction, and see Paragraph \ref{sss:BD} below for more details).

Finally, making use of \eqref{eq:mass} and of the specific form of $K[\rho,\nabla\rho]$, it is easy to check that we arrive at the balance law
\begin{equation} \label{eq:momentum}
\d_t\left(\rho u\right)\,+\,\div\bigl(\rho u\otimes u\bigr)\,+\,\,\nabla\Pi(\rho)\,-\,\nu\,\div\bigl(\rho Du\bigr)\,-\,\rho\nabla\Delta\rho\,+\,\mf{C}(\rho,u)\,=\,0\,,
\end{equation}
which expresses the conservation of momentum.

\begin{rem} \label{r:phase-b_2}
Notice that we can read system \eqref{eq:mass}-\eqref{eq:momentum} as a dynamical description
of propagation of interfaces and phase boundaries. Recall also the discussions in Subsections \ref{ss:i_cap} and \ref{ss:korteweg}, and in particular Remark \ref{r:phase-b}.
\end{rem}

%%%%%%%%%%
\subsection{The singular perturbation viewpoint} \label{ss:singular}

We have presented above the equations of motion we are interested in, namely \eqref{eq:mass}-\eqref{eq:momentum}.
Here we want to introduce relevant physical adimensional parameters in the equations, namely the Rossby, Mach and Weber numbers. Sending these parameters to $0$ will allow us to perform an
asymptotic analysis of the fast rotation and incompressible limits, in the regime of constant capillarity.

Let us first explain the main motivations for a singular perturbation study.

\subsubsection{Motivations} \label{sss:sing_motiv}

To assess the importance of some process in a particular situation, it is common practice in Physics to introduce and compare dimensional quantities
expressing the orders of magnitude of the variables under consideration, and to ignore the ``small'' ones.
More precisely, performing a rescaling in the equations allows to select dimensionless parameters: sending them to $0$ or to $+\infty$ leads to some simplifications
in the equations and possibly to a better understanding of the main phenomena which take place. The limit process
means that we are overlooking some negligible features of the experience, or focusing on some special ones.

Scale analysis is an efficient tool exploited both theoretically and in numerical experiments: it allows to reduce the complexity of the physical system under consideration, selecting just
relevant quantities, or the computational cost of simulations. In addition, it reveals to be very fruitful in real world applications.
This approach is expecially relevant for large-scale processes, like geophysical flows we want to consider here. 
Actually most, if not all mathematical models used in fluid mechanics (like e.g. incompressible Navier-Stokes equations) rely on an asymptotic analysis of
more complicated systems.

Nonetheless, a huge part of the available literature on scale analysis is based on formal asymptotic limits of (supposed to exist) solutions with respect to one or more singular parameters.
The main goal of the mathematical approach to singular perturbation problems is to provide a rigorous justification of the limit model employed in the physical approximation, and possibly to
capture corrections (the equations being non-linear in general, amplification of small quantities might occur) or further underlying effects (like e.g. the way the
quantities which are negligible in a certain regime are filtered off in the asymptotic limit).

\subsubsection{The Rossby, Mach and Weber numbers} \label{sss:sing_numb}

In view of the previous discussion, in the present paragraph our aim is to identify the physical parameters which are relevant with respect to the phenomena
we want to capture.

Let us consider rotation first. To establish its importance, a \textsl{na\"if} approach may be to compare time scales, and hence the time $\tau$ of one revolution
to the life $T$ of the entire physical process. But it is easy to realize that in geophysical flows also lenght scale can play an important role. Therefore, it is customary to take
rather the ratio between the time of one revolution and the time required to cover a distance $L$ at speed $U$.
More precisely, denoted by $L$ the typical lenght scale, by $U$ the typical velocity scale and by $\Theta$ the rotation rate (namely, $\Theta\,=\,2\pi/\tau$),
we define the \emph{Rossby number} as the ratio
$$
Ro\,:=\,\frac{U}{\Theta\,L}\,.
%Ro\;:=\;U\,/\,\bigl(\Theta\,L\bigr)\,.
$$
It compares advection to Corilolis force: rotation plays an important role whenever $Ro\,\lesssim\,1$ (see also Sections 1.5 and 3.6 of \cite{CR}).

A common simplification in geophysical fluid dynamics (see e.g. \cite{CR} and \cite{Ped}) is the Boussinesq approximation, which states the incompressibility of the velocity
field. Here, we generlize this assumption (as it is quite usual, see e.g. \cite{F-G-N}-\cite{F-G-GV-N}) by considering slightly compressible fluids.
The compressibility is measured by the so-called \emph{Mach number}, defined as
$$
Ma\,:=\,\frac{U}{c}\,,
$$
where $c$ denotes the speed of sound in the medium: more $Ma$ is small, more the fluid behaves as incompressible. This number will appear in front of the pressure term.
We refer also to \cite{F-N} for further details.

Finally, to evaluate the importance of capillarity, we compare typical inertial forces to stabilizing molecular cohesive forces, or (which is the same) kinetic energy to surface
tension energy. Hence, we introduce the \emph{Weber number} as
$$
We\,:=\,\frac{\rho_0\,U^2\,L}{\sigma}\,,
$$
where we set $\rho_0$ to be the reference density of the fluid and $\sigma$ the typical surface tension. Having $We$ small means that capillarity forces are relevant
in the ongoing physical process.

\subsubsection{Introducing the scaling in the equations} \label{sss:sing_eq}

We now introduce the previous parameters in the mathematical model, getting a rescaled system.
The mass equation \eqref{eq:mass} is not modified by this scale analysis: the Rossby, Mach and Weber numbers come into play in the momentum balance only.
More precisely, equation \eqref{eq:momentum} becomes
$$
\d_t\left(\rho u\right)\,+\,\div\bigl(\rho u\otimes u\bigr)\,+\,
\dfrac{1}{Ma\,^2}\,\nabla\Pi(\rho)\,-\,\nu\,\div\bigl(\rho Du\bigr)\,-\,
\dfrac{1}{We}\,\rho\nabla\Delta\rho\,+\,\dfrac{1}{Ro}\,\mf{C}(\rho,u)\,=\,0\,.
$$

As a final step, we take the point of view of the singular perturbation analysis. Namely, we fix a parameter
$\veps\,\in\,]0,1]$, and we set
$$
Ma\,=\,Ro\,=\,\veps\qquad\qquad \mbox{ and }\qquad\qquad We\,=\,\veps^2\,.
$$
Accordingly, the previous momentum equation becomes
\begin{equation} \label{eq:mom-sing}
\d_t\left(\rho u\right)\,+\,\div\bigl(\rho u\otimes u\bigr)\,+\,\dfrac{1}{\veps^2}\,\nabla\Pi(\rho)\,-\,\nu\,\div\bigl(\rho Du\bigr)\,-\,
\dfrac{1}{\veps^2}\,\rho\nabla\Delta\rho\,+\,\dfrac{1}{\veps}\,\mf{C}(\rho,u)\,=\,0\,.
\end{equation}

For all $\veps$ fixed, thanks to previous results on Navier-Stokes-Korteweg system (Coriolis force can be easily handled at this level),
we know the existence of a weak solution $(\rho_\veps,u_\veps)$ to equations \eqref{eq:mass}-\eqref{eq:mom-sing}.
Our mathematical problem is then to study the asymptotic behaviour of the family $\bigl(\rho_\veps,u_\veps\bigr)_\veps$
in the limit $\veps\,\ra\,0$. We recall that this corresponds to making the incompressible and fast rotation limit together, in the regime of constant capillarity.

Indeed, we remark that the same system could be obtained from \eqref{eq:mass}-\eqref{eq:momentum} using also a different approach (see \cite{J-L-W} and \cite{F_2015}),
in order to investigate the long-time behaviour of solutions.
Namely, one can perform the scaling $t\mapsto\veps t$, $u\mapsto\veps u$ and
$\nu\mapsto\veps\nu$, set the capillarity coefficient $k=\veps^{2\alpha}$, for some $0\leq\alpha\leq1$, and divide everything by $\veps^2$ (i.e. the highest power of $\veps$).
One finally bumps exactly into \eqref{eq:mom-sing} when making the choice $\alpha=0$, which means $k=1$: this is why we call this scaling the ``constant capillarity regime''.
Taking different values of $\alpha$ was investigated in \cite{F_2015}: we refer to this paper for results in both the regimes of constant and vanishing capillarity,
for the whole range of vanishing rates (namely, for any fixed $\alpha\in[0,1]$).

We conclude this part by pointing out that different scalings were allowed, in principle, for the Mach and Rossby numbers.
However, on the one hand our choice $Ma=Ro=\veps$ is in accordance with previous related works, see e.g. \cite{F-G-N}, \cite{F-G-GV-N}. This is the right scaling
in order to recover the \emph{geostrophic balance}, i.e. the equilibrium between strong Coriolis force and pressure (see e.g. \cite{CR}, \cite{Ped}, \cite{SR}).

On the other hand, we have to remark the strong analogy between our equations \eqref{eq:mass}-\eqref{eq:momentum} and the viscous shallow water system, studied in e.g. \cite{B-D_2003}-\cite{B-D-Met}
(which however makes sense only in $2$-D, because it derives when averaging vertical motions). We also refer to \cite{G}-\cite{G-SR_Mem} for a related system,
relevant in the investigation of the betaplane model for equatorial dynamics, where however the viscosity coefficients are taken to be constant.
In shallow water equations, the term corresponding to the pressure (density is actually replaced by the depth variation function) is rescaled according to
the so-called \emph{Froude number}: if $N$ denotes the stratification frequency (we have the relation $N^2\,\sim\,\d_3\rho$) and $H$ the typical depth of the domain, we define
$$
Fr\,:=\,\frac{U}{N\,H}\,.
$$
This number measures the importance of vertical stratification with respect to the (main) horizontal flow.
Pysical considerations reveal that for large-scale motions, in general, the relation $Fr\,^2\,\lesssim\,Ro$ must hold; moreover, a special regime
occurs when rotation and stratification effects are comparable, namely when $Fr\,\sim\,Ro$, a case which corresponds to take the so-called \emph{Burger number} proportional to $1$.
We refer to Chapter 9 of \cite{CR} for more details about this point.
By the parallel between our system and the shallow water system, there is a strict correspondence between $Fr$ and $Ma$: taking then the scaling $Ma=Ro=\veps$ is a way to
capture such a special regime when $Fr\,\sim\,Ro$.

%%%%%%%%%%%%%%%%%%%%%%%%%%%%%%%%%%%%%%%%%%%%%%%%%%
%%%%%%%%%%%%%%%%%%%%%%%%%%%%%%%%%%%%%%%%%%%%%%%%%%
\section{Mathematical properties of the model} \label{s:math_prop}

In the present section we want to give some insights on the main mathematical features our model enjoys.
These properties are quite general, since they are determined by capillarity on the one side and by Earth rotation on the other side. Nonetheless,
for clarity of exposition and for sake of conciseness, we prefer to focus on our special working setting: let us fix our assumptions first.

%%%%%
\subsection{Fixing the working hypotheses} \label{ss:hypotheses}
Let us recall here our working setting, and introduce further important assumptions in order to tackle the mathematical problem.

We fix the domain $\Omega\,:=\,\R^2\,\times\,\,]0,1[\,$. Recall that, accordingly, we split $x\in\Omega$ as $x=(x^h,x^3)$. 
In $\R_+\times\Omega$, and for a small parameter $0<\veps\leq1$, let us consider the rescaled Navier-Stokes-Korteweg system
\begin{equation} \label{eq:NSK-sing}
\begin{cases}
\d_t\rho\,+\,\div\left(\rho u\right)\,=\,0 \\[1ex]
\d_t\left(\rho u\right)\,+\,\div\bigl(\rho u\otimes u\bigr)\,+\,
\dfrac{1}{\veps^2}\,\nabla\Pi(\rho)\,-\,\nu\,\div\bigl(\rho Du\bigr)\,-\,
\dfrac{1}{\veps^2}\,\rho\nabla\Delta\rho\,+\,\dfrac{1}{\veps}\,\mf{C}(\rho,u)\,=\,0\,.
\end{cases}
\end{equation}

Here, we suppose that the Coriolis operator $\mf C$ is given by formula \eqref{eq:coriolis},
where $\mf{c}$ is a scalar function of the horizontal variables only. For the time being, let us assume
$\mf c\,\in\,W^{1,\infty}(\R^2)$: this hypothesis will allow us to establish uniform bounds in Subsection \ref{ss:cap-effects}.
See also Theorem \ref{t:weak} and the subsequent remark in Subsection \ref{ss:weak}

In addition, we assume that the pressure $\Pi(\rho)$ can be decomposed into $\Pi\,=\,P\,+\,P_c$, where $P$ is the classical
pressure law, given by the Boyle relation
\begin{equation} \label{eq:def_P}
P(\rho)\,:=\,\frac{1}{2\g}\,\rho^\g\,,\qquad\qquad\mbox{ for some }\quad1\,<\,\g\,\leq\,2\,,
\end{equation}
and the second term is the so-called \emph{cold pressure} component, for which we take the expression
\begin{equation} \label{eq:def_cold}
P_c(\rho)\,:=\,-\,\frac{1}{2\g_c}\,\rho^{-\g_c}\,,\qquad\qquad\mbox{ with }\qquad 1\,\leq\,\g_c\,\leq\,2\,.
\end{equation}
For us, $\g_c=2$ always. The presence of the $1/2$ is just a normalization in order to have $\Pi'(1)=1$: this fact simplifies some computations in the sequel.
Having a cold component in the pressure law is important in the theory of existence of weak solutions (this is our motivation in assuming it), but it plays no special role in the singular
perturbation analysis.

As said in Subsection \ref{ss:motion}, we supplement system \eqref{eq:NSK-sing} by complete slip boundary conditions,
in order to avoid the appearing of boundary layers effects.
Namely, if $n$ denotes the unitary outward normal to the boundary $\d\Omega$ of the domain (simply,
$\d\Omega=\{x^3=0\}\cup\{x^3=1\}$), we impose
\begin{equation} \label{eq:bc}
\left(u\cdot n\right)_{|\d\Omega}\,=\,0\,,\qquad
\left(\nabla\rho\cdot n\right)_{|\d\Omega}\,=\,0\,,\qquad
\bigl((Du)n\times n\bigr)_{|\d\Omega}\,=\,0\,.
\end{equation}

\begin{rem} \label{r:period-bc}
Equations \eqref{eq:NSK-sing}, supplemented by boundary conditions \eqref{eq:bc},
can be recasted as a periodic problem with respect to the vertical variable, in the new domain
$$
\wtilde \Omega\,=\,\R^2\,\times\,\mbb{T}^1\,,\qquad\qquad\mbox{ with }\qquad\mbb{T}^1\,:=\,[-1,1]/\sim\,,
$$
where $\sim$ denotes the equivalence relation which identifies $-1$ and $1$.
Indeed, the equations are invariant if we extend $\rho$ and $u^h$ as even functions with respect to $x^3$, and $u^3$ as an odd function.

In what follows, we will always assume that such modifications have been performed on the initial data, and
that the respective solutions keep the same symmetry properties.
\end{rem}

Here we will always consider initial data $(\rho_0,u_0)$ such that $\rho_0\geq0$ and
\begin{equation} \label{eq:initial}
\begin{cases}
\dfrac{1}{\veps}\left(\rho_0\,-\,1\right)\;\in\;L^\g(\Omega)\qquad\mbox{ and }\qquad
\dfrac{1}{\veps}\left(\dfrac{1}{\rho_0}\,-\,1\right)\;\in\;L^{\g_c}(\Omega) \\[1ex]
\sqrt{\rho_0}\,u_0\;,\;\nabla\sqrt{\rho_0}\;,\;\dfrac{1}{\veps}\nabla\rho_0\quad\in\;L^2(\Omega)\,.
\end{cases}
\end{equation}
When taking a family $\bigl(\rho_{0,\veps},u_{0,\veps}\bigr)_\veps$ of initial data, we will require that conditions \eqref{eq:initial} hold true for any
$\veps\in\,]0,1]$, and \emph{uniformly} in $\veps$ (see also the hypotheses at the beginning of Subsection \ref{ss:var-result}).

At this point, we introduce the internal energy functions $h(\rho)$ and $h_c(\rho)$ in the following way: we require that
$$
\begin{cases}
h''(\rho)\,=\,\dfrac{P'(\rho)}{\rho}\,=\,\rho^{\g-2}\qquad\qquad\mbox{ and }\qquad\qquad
h(1)\,=\,h'(1)\,=\,0\,, \\[2ex]
h_c''(\rho)\,=\,\dfrac{P_c'(\rho)}{\rho}\,=\,\rho^{-\g_c-2}\qquad\qquad\mbox{ and }\qquad\qquad
h_c(1)\,=\,h_c'(1)\,=\,0\,.
\end{cases}
$$
Let us define the classical energy %for $\g_c=2$,
\begin{equation}  \label{eq:def_E_2} %\label{eq:def_E_2}
% E_2[\rho,u](t)
E_\veps[\rho,u](t) \;:=\;\int_\Omega\left(\dfrac{1}{\veps^2}\,h(\rho)\,+\,\dfrac{1}{\veps^2}\,h_c(\rho)\,+\,
\dfrac{1}{2}\,\rho\,|u|^2\,+\,
\dfrac{1}{2\,{\veps^2}}\,|\nabla\rho|^2\right)dx\,,
\end{equation}
and the BD entropy function
\begin{equation} \label{eq:def_F}
F_\veps[\rho](t)\;:=\;\frac{\nu^2}{2}\int_\Omega\rho\,|\nabla\log\rho|^2\,dx\;=\;
2\,\nu^2\int_\Omega\left|\nabla\sqrt{\rho}\right|^2\,dx\,.
\end{equation}
Finally, let us denote by $E_\veps[\rho_0,u_0]\,\equiv\,E_\veps[\rho,u](0)$ and by
$F_\veps[\rho_0]\,\equiv\,F_\veps[\rho](0)$ the same energies, when computed on the initial data $\bigl(\rho_0,u_0\bigr)$.

\begin{rem} \label{r:en_initial}
Notice that, under our assumptions, we deduce the existence of a ``universal constant'' $C_0>0$ such that
$$
E_\veps[\rho_{0,\veps}\,,\,u_{0,\veps}]\,+\,F_\veps[\rho_{0,\veps}]\,\leq\,C_0\,.
$$
\end{rem}

%%%%%
\subsection{Features due to capillarity} \label{ss:cap-effects}

We now enter more in detail in the description of the properties of our model. Here we look at capillarity effects.
The main point is to establish energy estimates (in Paragraph \ref{sss:BD}), which will be of two kinds: classical energy estimates and BD entropy estimates.

\subsubsection{On the BD entropy structure} \label{sss:BD}

As we have already remarked in the Introduction, the combination of capillarity and density-dependent viscosity coefficient gives a special mathematical
structure to our system, the so-called \emph{BD structure}, which consists of a second energy conservation and which in turn allows to exploit the presence of the Korteweg tensor
in the equations, proving higher order regularity estimates for $\rho$.
This additional regularity is fundamental both in the theory of weak solutions and in the study of the singular perturbation problem.

The BD entropy structure was first discovered by Bresch and Desjardins for Korteweg models (see also paper \cite{B-D-L}), but it was then generalized by the same authors to
several other systems having density-dependent viscosity coefficients.
More precisely, the BD structure is a fundamental tool in weak solutions theory for compressible fluids with degenerate (i.e. vanishing with vacuum) viscosity
coefficients. For further details, the reader is referred to \cite{B-D-Zat} and the references therein (see also the introductions of \cite{F_2015}-\cite{F_2016}).

\medbreak
We now present energy estimates and additional uniform bounds for the family of solutions $\bigl(\rho_\veps,u_\veps\bigr)_\veps$ to system \eqref{eq:NSK-sing}.
Since the Coriolis term makes no work, classical energy estimates are easy to find. This is not the case for the BD entropy: the main point is then to control
the rotation term uniformly in $\veps$. We refer to papers \cite{F_2015}-\cite{F_2016} for the proofs of the estimates.
Of course, the precise computations are rigorous for smooth enough solutions: see Subsection \ref{ss:weak} for more details.

We start by proving classical energy estimates.
\begin{prop} \label{p:E}
Let $(\rho,u)$ be a smooth solution to system \eqref{eq:NSK-sing} in $\R_+\times\Omega$, related to the initial datum $\bigl(\rho_0,u_0\bigr)$.

Then, for all $\veps>0$ and all $t\in\R_+$, one has
$$
\frac{d}{dt}E_\veps[\rho,u]\,+\,\nu\int_\Omega\rho\,|Du|^2\,dx\,=\,0\,.
$$
\end{prop}

From the previous statement, we deduce the first class of uniform bounds.
\begin{coroll} \label{c:E}
Let $\g_c=2$, and let $\bigl(\rho_\veps,u_\veps\bigr)_\veps$ be a family of smooth solutions to system \eqref{eq:NSK-sing} in $\R_+\times\Omega$, related to
initial data $\bigl(\rho_{0,\veps}\,,\,u_{0,\veps}\bigr)_\veps$, and assume that the initial energy $E_\veps[\rho_{0,\veps}\,,\,u_{0,\veps}]<+\infty$.

Then one has the following properties, uniformly in $\veps\in\,]0,1]$:
$$
\sqrt{\rho_\veps}\,u_\veps\;,\;\frac{1}{\veps}\,\nabla\rho_\veps\quad\in\;L^\infty\bigl(\R_+;L^2(\Omega)\bigr)\quad\mbox{ and }\quad
\sqrt{\rho_\veps}\,Du_\veps\;\in\;L^2\bigl(\R_+;L^2(\Omega)\bigr)\,.
$$
Moreover, we also get
$$
\frac{1}{\veps}\left(\rho_\veps\,-\,1\right)\,\in\,L^\infty\bigl(\R_+;L^\g(\Omega)\bigr)\qquad\mbox{ and }\qquad
\frac{1}{\veps}\left(\frac{1}{\rho_\veps}\,-\,1\right)\,\in\,L^\infty\bigl(\R_+;L^2(\Omega)\bigr)\,.
$$
\end{coroll}

\begin{rem} \label{r:rho_L^2}
In particular, under our assumptions we have
$$ 
\left\|\rho_\veps\,-\,1\right\|_{L^\infty(\R_+;L^2(\Omega))}\,\leq\,C\,\veps\,.
$$
\end{rem}

BD entropy estimates are harder to establish. Let us present the final estimate, without giving the details on how controlling the Coriolis force uniformly in $\veps$.
We remark that at this point the hypothesis $\nabla_h\mf c\in L^{\infty}$ comes into play in the proof (see also Remark \ref{r:existence} below).

\begin{prop} \label{p:F}
Let $(\rho,u)$ be a smooth solution to system \eqref{eq:NSK-sing} in $\R_+\times\Omega$, related to the initial
datum $\bigl(\rho_0,u_0\bigr)$.

Then, there exists a positive constant $C$, such that, for all $T\in\R_+$ fixed, one has
$$ %\begin{eqnarray*}
\sup_{t\in[0,T]}F[\rho](t)\,+\,\frac{\nu}{\veps^2}\int^T_0\!\!\int_\Omega\left|\nabla^2\rho\right|^2\,dx\,dt\,+\,
\frac{\nu}{\veps^2}\int^T_0\!\!\int_\Omega\Pi'(\rho)\,\left|\nabla\sqrt{\rho}\right|^2\,dx\,dt\,\leq\,C\,(1+T)\,.
$$ %\end{eqnarray*}
The constant $C$ depends just on the viscosity coefficient $\nu$ and on the energies of the initial data $E_\veps[\rho_0,u_0]$ and $F_\veps[\rho_0]$.
\end{prop}

Let us point out here that the pressure term can be also written as
$$ %\begin{eqnarray}
\frac{\nu}{\veps^2}\int_\Omega\Pi'(\rho)\,\left|\nabla\sqrt{\rho}\right|^2\,=\,
\frac{C_\g\,\nu}{4\,\veps^2}\int_\Omega\left|\nabla\left(\rho^{\g/2}\right)\right|^2\,+\,\frac{C_{\g_c}\,\nu}{4\,\veps^2}\int_\Omega\left|\nabla\left(\rho^{-\g_c/2}\right)\right|^2\,,
$$ %\end{eqnarray}
for some positive constants $C_\g$ and $C_{\g_c}$. In particular, when $\g_c=2$ (which will be our case), $C_{\g_c}=1$.

From Proposition \ref{p:F} we infer the following consequences.
\begin{coroll} \label{c:F}
Let $\g_c=2$, and let $\bigl(\rho_\veps,u_\veps\bigr)_\veps$ be a family of smooth solutions to system \eqref{eq:NSK-sing} in $\R_+\times\Omega$, related to
initial data $\bigl(\rho_{0,\veps}\,,\,u_{0,\veps}\bigr)_\veps$, and assume that the initial energy $E_\veps[\rho_{0,\veps}\,,\,u_{0,\veps}]<+\infty$.

Then one has the following bounds, uniformly for $\veps>0$:
$$
\begin{cases}
\nabla\sqrt{\rho_\veps}\;\in\;L^\infty_{loc}\bigl(\R_+;L^2(\Omega)\bigr) \\[1ex]
\dfrac{1}{\veps}\,\nabla^2\rho_\veps\;,\quad\dfrac{1}{\veps}\,\nabla\!\left(\rho_\veps^{\g/2}\right)\;,\quad
\dfrac{1}{\veps}\,\nabla\!\left(\dfrac{1}{\rho_\veps}\right)\qquad \in\;L^2_{loc}\bigl(\R_+;L^2(\Omega)\bigr)\,.
\end{cases}
$$
In particular, the family $\bigl(\veps^{-1}\,(\rho_\veps-1)\bigr)_\veps$ is bounded in $L^p_{loc}\bigl(\R_+;L^\infty(\Omega)\bigr)$
for any $2\leq p<4$.
\end{coroll}

\subsubsection{Additional bounds} \label{sss:bounds}

Let us continue and deduce further uniform bounds on a family $\bigl(\rho_\veps,u_\veps\bigr)_\veps$ of solutions to system \eqref{eq:NSK-sing} in $\R_+\times\Omega$, related to
initial data $\bigl(\rho_{0,\veps}\,,\,u_{0,\veps}\bigr)_\veps$ such that $E_\veps[\rho_{0,\veps}\,,\,u_{0,\veps}]<+\infty$.

First of all, working a little bit one can establish, uniformly in $\veps$:
\begin{eqnarray*}
\bigl(u_\veps\bigr)_\veps & \subset & L^\infty_T\bigl(L^2\bigr)+L^2_T\bigl(L^{3/2}\bigr)\,\hookrightarrow\,
L^2_T\bigl(L^{3/2}_{loc}\bigr) \\ %\label{sing-b:u} \\
\bigl(Du_\veps\bigr)_\veps & \subset & \Bigl(L^2_T\bigl(L^2\bigr)
+L^1_T\bigl(L^{3/2}\bigr)\Bigr)\,\cap\,\Bigl(L^2_T\bigl(L^2+L^1\bigr)\Bigr)\,. %\label{sing-b:Du}
\end{eqnarray*}
In particular, $\bigl(Du_\veps\bigr)_\veps$ is uniformly bounded in $L^1_T\bigl(L^{3/2}_{loc}\bigr)$; therefore,
by Sobolev embeddings we gather also the additional continuous inclusion $\bigl(u_\veps\bigr)_\veps\,\subset\,L^1_T\bigl(L^3_{loc}\bigr)$.

Furthermore, we also infer the uniform bounds
\begin{eqnarray}
\bigl(\rho_\veps\,u_\veps\bigr)_\veps & \subset & L^\infty_T\bigl(L^2+L^{3/2}\bigr)\,\cap\,
\Bigl(L^\infty_T\bigl(L^2\bigr)+L^2_T\bigl(L^2\bigr)\Bigr) \nonumber \\%\label{sing-b:rho-u} \\
\bigl(D(\rho_\veps\,u_\veps)\bigr)_\veps & \subset & L^2_T\bigl(L^2+L^{3/2}\bigr)\,+\,L^\infty_T\bigl(L^1\bigr)\;\hra\;
L^2_T\bigl(L^1_{loc}\bigr)\,. \label{sing-b:D_rho-u}
\end{eqnarray}
In particular, we deduce that $\bigl(\rho_\veps\,u_\veps\bigr)_\veps$ is a bounded family in 
$L^\infty_T\bigl(L^{3/2}_{loc}\bigr)\,\cap\,L^2_T\bigl(L^2\bigr)$.

For the sake of completeness let us also establish uniform bounds on quantities related to $\rho^{3/2}_\veps\,u_\veps$.
First of all, we get
$$
\bigl(\rho^{3/2}_\veps\,u_\veps\bigr)_\veps\;\subset\;L^\infty_T\bigl(L^2+L^{3/2}\bigr)\,\cap\,
\Bigl(L^\infty_T\bigl(L^2\bigr)+L^2_T\bigl(L^2\bigr)\Bigr)\,;
$$
on the other hand, we have also the uniform embedding
\begin{equation} \label{sing-b:D(rho-u)}
\Bigl(D\bigl(\rho^{3/2}_\veps\,u_\veps\bigr)\Bigr)_\veps\;\subset\;L^2_T\bigl(L^2(\Omega)+L^{3/2}(\Omega)\bigr)\,.
\end{equation}
Therefore (see Theorem 2.40 of \cite{B-C-D}), we infer that
$\bigl(\rho^{3/2}_\veps\,u_\veps\bigr)_\veps$ is uniformly bounded in $L^2_T\bigl(L^3(\Omega)\bigr)$.

Finally, from this fact combined with the usual decomposition $\sqrt{\rho_\veps}\,=\,1+(\sqrt{\rho_\veps}-1)$
and Sobolev embeddings, it follows also that
$$ %\begin{equation} \label{sing-b:rho^2-u)}
\bigl(\rho_\veps^2\,u_\veps\bigr)_\veps\;\subset\;L^2_T\bigl(L^2(\Omega)\bigr)\,.
$$ %\end{equation}

\subsection{Existence of weak solutions} \label{ss:weak}

At this point, let us spend a few words about the global in time existence of weak solutions to our system.

First of all, let us recall the definition. The integrability properties are justified by previous energy estimates.
\begin{defin} \label{d:weak}
Fix initial data $(\rho_0,u_0)$ satisfying the conditions in \eqref{eq:initial}, with $\rho_0\geq0$.

We say that $\bigl(\rho,u\bigr)$ is a \emph{weak solution} to system \eqref{eq:NSK-sing}-\eqref{eq:bc}
in $[0,T[\,\times\Omega$ (for some $T>0$) with initial datum $(\rho_0,u_0)$ if the following conditions are verified:
\begin{itemize}
 \item[(i)] $\rho\geq0$ almost everywhere, and one has that
 $\veps^{-1}\bigl(\rho-1\bigr)\,\in\,L^\infty\bigl([0,T[\,;L^\g(\Omega)\bigr)$,
%${\rm Fr}^{-1}\bigl(1/\rho-1\bigr)\,\in\,L^\infty\bigl([0,T[\,;L^{\g_c}(\Omega)\,\cap\,L^1(\Omega)\bigr)$, ${\rm We}^{-1/2}\nabla\rho$
$\veps^{-1}\bigl(1/\rho-1\bigr)\,\in\,L^\infty\bigl([0,T[\,;L^{\g_c}(\Omega)\bigr)$, $\veps^{-1}\nabla\rho$
and $\nabla\sqrt{\rho}\;\in L^\infty\bigl([0,T[\,;L^2(\Omega)\bigr)$ and
$\veps^{-1}\nabla^2\rho\in L^2\bigl([0,T[\,;L^2(\Omega)\bigr)$;
\item[(ii)] $\sqrt{\rho}\,u\,\in L^\infty\bigl([0,T[\,;L^2(\Omega)\bigr)$ and $\sqrt{\rho}\,Du\,\in L^2\bigl([0,T[\,;L^2(\Omega)\bigr)$;
\item[(iii)] the mass and momentum equations are satisfied in the weak sense: for any scalar function
$\phi\in\mc{D}\bigl([0,T[\,\times\Omega\bigr)$ one has the equality
$$%\begin{equation} \label{eq:weak-mass}
-\int^T_0\int_\Omega\biggl(\rho\,\d_t\phi\,+\,\rho\,u\,\cdot\,\nabla\phi\biggr)\,dx\,dt\,=\,\int_\Omega\rho_0\,\phi(0)\,dx\,,
$$%\end{equation}
and for any vector-field $\psi\in\mc{D}\bigl([0,T[\times\Omega;\R^3\bigr)$ one has
\begin{eqnarray}
& & \hspace{-0.3cm}
\int_\Omega\rho_0u_0\cdot\psi(0)dx\,=\,\int^T_0\!\!\int_\Omega\biggl(-\rho u\cdot\d_t\psi\,-\,\rho u\otimes u:\nabla\psi\,-\,
\frac{1}{\veps^2}\Pi(\rho)\div\psi\,+ \label{eq:weak-momentum} \\
& & \qquad +\,\nu \rho Du:\nabla\psi\,+\,\frac{1}{\veps^2}\rho\Delta\rho\div\psi\,+\,
\frac{1}{\veps^2}\Delta\rho\nabla\rho\cdot\psi\,+\,\frac{\mf{c}(x^h)}{\veps}e^3\times\rho u\cdot\psi\biggr)\,dx\,dt\,; \nonumber
\end{eqnarray}
\item[(iv)] for almost every $t\in\,]0,T[\,$, the following energy inequalities hold true:
\begin{eqnarray} 
& & \hspace{-0.5cm} E_\veps[\rho,u](t)\,+\,\nu\int^t_0\int_\Omega\rho\,\left|Du\right|^2\,dx\,d\tau\;\leq\;E_\veps[\rho_0,u_0] \label{en_est:E} \\
& & \hspace{-0.5cm} F_\veps[\rho](t)\,+\,\dfrac{\nu}{\veps^2}\int^t_0\int_\Omega \Pi'(\rho)\,|\nabla\sqrt{\rho}|^2\,dx\,d\tau\,+\,
\dfrac{\nu}{\veps^{2}}\int^t_0\int_\Omega\left|\nabla^2\rho\right|^2\,dx\,d\tau\;\leq\;C\,(1+T)\,, \label{en_est:F}
\end{eqnarray}
for some constant $C$ depending just on $\bigl(E_\veps[\rho_{0},u_0],F_\veps[\rho_{0}],\nu\bigr)$.
\end{itemize}
\end{defin}

The last condition in the definition is required just in the study of the singular perturbation problem, because we have to assume \textsl{a priori} that
the family of weak solutions satisfies relevant uniform bonds.

The existence of weak solutions, in the sense of the previous definition, is guaranteed by the next result.
\begin{thm} \label{t:weak}
Let $\g_c=2$ in \eqref{eq:def_cold} and $\mf c\,\in\,W^{1,\infty}(\R^2)$ in \eqref{eq:coriolis}.
For any fixed $\veps>0$, consider a couple $(\rho_0,u_0)$ satisfying conditions \eqref{eq:initial}, with $\rho_0\geq0$.

Then, there exits a global in time weak solution $(\rho,u)$ to system \eqref{eq:NSK-sing}, %in the sense of Definition \ref{d:weak},
related to the initial datum $(\rho_0,u_0)$. 
\end{thm}

\begin{rem} \label{r:existence}
\begin{itemize}
 \item The hypothesis $\g_c=2$ is assumed just for simplicity here, but it is not really necessary for existence (see also comments below).
 \item The condition $\mf{c}\,\in\,W^{1,\infty}$ is important in order to take advantage of the BD entropy structure of our system.
However, it can be deeply relaxed at this level, in presence of further assumptions (e.g. when friction terms are considered, as in paper \cite{B-D_2003}).
\end{itemize}
\end{rem}

The previous result can be established arguing exactly as in \cite{B-D-Zat}. Actually, the statement of \cite{B-D-Zat} holds true under more general assumptions
than ours (as for the cold component of the pressure and the viscosity coefficient, for instance). We also refer to \cite{Mu-Po-Zat}-\cite{Mu-Po-Zat_2015}.

So, we omit the proof. It is based on a quite standard approximation procedure and passage to the limit: we construct smooth solutions to a perturbed system,
which however preserves the mathematical structure of the original one. We can then deduce existence of a weak solution to our system \eqref{eq:NSK-sing} as a limit of the previous family of
smooth approximate solutions.
In particular, this construction allows to justify the integrability properties we require in Definition \ref{d:weak}, see points (i), (ii) and (iv):
as a matter of fact, uniform bounds which follow from energy estimates, established in Subsection \ref{ss:cap-effects}, will be inherited also by the weak solutions.

%%%%%
\subsection{Effects due to fast rotation} \label{ss:rot-effects}

In the previous parts, we have proved the existence of weak solutions for our system, and the relevant bounds they enjoy.
In this subsection we identify properties related to the effects of a strong Coriolis force.
Namely, we will establish the analogue of the \emph{Taylor-Proudman theorem} in our context. On the other hand, as already pointed out, we ignore boundary layer effects.
Propagation of waves will be treated instead in Section \ref{s:proof}.

For doing this, let us fix a family $\bigl(\rho_\veps,u_\veps\bigr)_\veps$ of weak solutions to system \eqref{eq:NSK-sing}-\eqref{eq:bc},
related to the initial data $\bigl(\rho_{0,\veps}\,,\,u_{0,\veps}\bigr)_\veps$ such that $E_\veps[\rho_{0,\veps}\,,\,u_{0,\veps}]<+\infty$.

First of all, by uniform bounds we immediately deduce that $\rho_\veps\,\ra\,1$ (strong convergence)
in $L^\infty\bigl(\R_+;H^1(\Omega)\bigr)\,\cap\,L^2_{loc}\bigl(\R_+;H^2(\Omega)\bigr)$, with convergence rate $O(\veps)$.
So, we can write $\rho_\veps\,=\,1\,+\,\veps\,r_\veps$, with the family $\bigl(r_\veps\bigr)_\veps$ bounded in
the previous spaces. Then we infer that
\begin{equation} \label{eq:conv-r}
\qquad
r_\veps\,\rightharpoonup\,r\qquad\qquad\qquad\mbox{ in }\quad
L^\infty\bigl(\R_+;H^1(\Omega)\bigr)\,\cap\,L^2_{loc}\bigl(\R_+;H^2(\Omega)\bigr)\,.
\end{equation}

In the same way, if we define $a_\veps\,:=\,\bigl(1/\rho_\veps-1\bigr)/\veps$, we gather that $\bigl(a_\veps\bigr)_\veps$ is uniformly
bounded in $L^\infty\bigl(\R_+;L^2(\Omega)\bigr)\,\cap\,L^2_{loc}\bigl(\R_+;H^1(\Omega)\bigr)$. So it weakly converges
to some $a$ in this space: more precisely,
\begin{equation} \label{eq:conv-a}
\qquad
a_\veps\,\rightharpoonup\,-\,r\qquad\qquad\qquad\mbox{ in }\quad
L^\infty\bigl(\R_+;L^2(\Omega)\bigr)\,\cap\,L^2_{loc}\bigl(\R_+;H^1(\Omega)\bigr)\,.
\end{equation}

Again by uniform bounds, we also deduce
\begin{equation}
\qquad\qquad\qquad
u_\veps\,\rightharpoonup\,u\,\qquad\qquad\qquad\mbox{ in }\quad L^2_{loc}\bigl(\R_+;L^{3/2}_{loc}(\Omega)\bigr) \label{eq:conv-u}
\end{equation}
and $Du_\veps\,\rightharpoonup\,Du$ in $L^2_{loc}\bigl(\R_+;L^1_{loc}(\Omega)\bigr)$, where we have identified $(L^1)^*$ with
$L^\infty$.

Notice also that, by uniqueness of the limit, we have the additional properties
\begin{eqnarray*}
\sqrt{\rho_\veps}\,u_\veps\,\stackrel{*}{\rightharpoonup}\,u & \qquad\qquad\mbox{ in }\quad &  L^\infty\bigl(\R_+;L^2(\Omega)\bigr) \\[1ex]
\rho_\veps\,u_\veps\,\rightharpoonup\,u & \qquad\qquad\mbox{ in }\quad &  L^2_{loc}\bigl(\R_+;L^2(\Omega)\bigr) \\[1ex]
\sqrt{\rho_\veps}\,Du_\veps\,\rightharpoonup\,Du & \qquad\qquad\mbox{ in }\quad &  L^2\bigl(\R_+;L^2(\Omega)\bigr)\,,
\end{eqnarray*}
where $\stackrel{*}{\rightharpoonup}$ denotes the weak-$*$ convergence in $L^\infty\bigl(\R_+;L^2(\Omega)\bigr)$.

Finally, from the analysis of Paragraphs \ref{sss:BD} and \ref{sss:bounds}, we deduce some constraints the limit points $(r,u)$ have to satisfy.
This is exactly the analogue of the Taylor-Proudman theorem in our context.

\begin{prop} \label{p:TP}
Let $\bigl(\rho_\veps,u_\veps\bigr)_\veps$ be a family of weak solutions to system \eqref{eq:NSK-sing}-\eqref{eq:bc},
related to the initial data $\bigl(\rho_{0,\veps}\,,\,u_{0,\veps}\bigr)_\veps$ such that $E_\veps[\rho_{0,\veps}\,,\,u_{0,\veps}]<+\infty$.
Let us define $r_\veps:=\veps^{-1}\left(\rho_\veps-1\right)$, and let $(r,u)$ be a limit point of the sequence
$\bigl(r_\veps,u_\veps\bigr)_\veps$.

Then $r=r(x^h)$ and $u=\bigl(u^h(x^h),0\bigr)$, with $\div_{\!h}u^h=0$. Moreover, we have
$$
\mf{c}\,u^h\;=\,\nabla_h^\perp\bigl(\Id\,-\,\Delta_h\bigr)r\qquad\qquad\mbox{ and }\qquad\qquad
u^h\,\cdot\,\nabla_h\mf{c}\,\equiv\,0\,.
$$
\end{prop}

\begin{rem} \label{r:limit-dens}
Notice that, by the previous proposition and the fact that $r$ and $u$ belongs to $L^\infty\bigl(\R_+;L^2(\Omega)\bigr)$, we actually get
$r\in L^\infty\bigl(\R_+;H^3(\Omega)\bigr)$.
\end{rem}

The proof of the previous proposition relies in testing the mass and momentum equations respectively on smooth functions $\vphi$ and $\veps\psi$.
From the equations for the velocity fields we then obtain the (quasi-)geostrophic balance relation.

Departures from geostrophy, which are determined by components of the solutions which do not respect the conditions found in the previous proposition, arise as superposition of waves.
By Proposition \ref{p:TP}, the wave propagator, i.e. the singular perturbation operator, can be defined as 
\begin{equation} \label{eq:sing-op}
\begin{array}{lccc}
\mc{A}\,: & L^2(\Omega)\;\times\;L^2(\Omega) & \longrightarrow & H^{-1}(\Omega)\;\times\;H^{-3}(\Omega) \\
& \bigl(\,r\;,\;V\,\bigr) & \mapsto & \Bigl(\div V\;,\;\mf{c}(x^h)\,e^3\times V\,+\nabla\bigl(\Id-\Delta\bigr)r\Bigr)\,.
\end{array}
\end{equation}
Remark that $\mc A$ %is skew-adjoint with respect to the usual $L^2$ scalar product. Notice also that it 
has variable coefficients whenever $\mf c$ is a non-constant function.

%%%%%%%%%%%%%%%%%%%%%%%%%%%%%%%%%%%%%%%%%%%%%%%%%%%%%%%%%%%%%
\section{Asymptotic limits for capillary fluids in fast rotation} \label{s:asympt}

In the present section we % study a singular perturbation problem for the Navier-Stokes-Korteweg model with Coriolis force introduced above. Namely, we
perform the incompressible and fast rotation asymptotics simultaneously, while we keep the capillarity coefficient constant in
order to capture surface tension effects.

After spending a few words on the case of constant axis, we then specialize on the case of variable rotation axis.
We show the convergence of our system to a linear parabolic-type equation with variable coefficients.
Besides, here we look for minimal regularity assumptions on the variations of the axis, and we consider conditions of Zygmund-type for the function $\mf{c}$.

A more complete analysis and further results can be found in \cite{F_2015}-\cite{F_2016}.
%%%%%
\subsection{An overview of the constant rotation axis case} \label{ss:constant}

Let us consider the case $\mf c\equiv1$ first. This situation has been studied in \cite{F_2015}.

The analysis developed in Section \ref{s:math_prop} is almost enough to pass to the limit in the weak formulation of the equations,
when testing them on functions $(\vphi,\psi)$ in the kernel of the singular perturbation operator $\mc A$, defined in \eqref{eq:sing-op}.
The difficulty relies in passing to the limit in the convective and capillarity terms
$$
\rho_\veps\,u_\veps\otimes u_\veps\qquad\qquad\mbox{ and }\qquad\qquad \frac{1}{\veps^2}\,\Delta\rho\,\nabla\rho\,.
$$

As usual in singular perturbation problems, writing the equations (roughly speaking) in the form of a wave system
$$
\veps\,\d_tU_\veps\,+\,\mc{A}\,U_\veps\,=\,\veps\,F_\veps\,,
$$
one expects that components of the solutions in the kernel of $\mc A$, say $\mbb{P}U_\veps$, strongly converge (in suitable spaces) to a solution of the target system.
On the other hand, according to the previous equation, the projections onto the orthogonal complement of ${\rm Ker}\,\mc A$, say $\mbb{Q}U_\veps$, produce fast oscillations,
which therefore should weakly converge to $0$. In order to treat non-linearities, then, the first approach is to look for strong convergence properties, and especially in this case for
dispersion of the components in $\bigl({\rm Ker}\,\mc A\bigr)^\perp$.

The fact that $\mc A$ has constant coefficients for $\mf c\equiv1$ simplifies the study of dispersive properties, since spectral analysis tools are available.
Direct computations show that the point spectrum of $\mc A$ reduces to the only eigenvalue $0$: therefore, one wants to use the celebrated RAGE theorem to prove
dispersion of $\mbb{Q}U_\veps$ in suitable norms.

Nevertheless, RAGE theorem is not directly applicable, since $\mc A$ is not skew-adjoint with respect to the classical $L^2$ scalar product. Then, after having localized in frequencies,
the idea is to resort to microlocal symmetrization arguments, in order to define a scalar product with respect to which $\mc A$ is skew-adjoint, and then RAGE theorem can be applied.
We remark that the symmetrizer involves a loss of derivatives for the density component, but this loss is safely handled because of the frequency localization
(notice however that one disposes of additional regularity for the density, provided by BD entropy estimates).

Finally, without entering into the details (for instance, precise assumptions on the initial data will be made in Subsection \ref{ss:var-result} below),
one can prove a result of the following type.

\begin{thm} \label{th:sing}
Let $\bigl(\rho_\veps\,,\,u_\veps\bigr)_\veps$ be a family of weak solutions (in the sense of Definition \ref{d:weak} above) to system
\eqref{eq:NSK-sing}-\eqref{eq:bc} in $[0,T]\times\Omega$, related to suitable initial data
$\bigl(\rho_{0,\veps},u_{0,\veps}\bigr)_\veps$. Suppose that the symmetriy properties of Remark \ref{r:period-bc} are verified.
Define $r_\veps\,:=\,\veps^{-1}\left(\rho_\veps-1\right)$.

Then, up to the extraction of a subsequence, one has the convergence properties
\begin{itemize}
\item[(a)] ~$r_\veps\,\rightharpoonup\,r$ in $L^\infty\bigl([0,T];H^1(\Omega)\bigr)\,\cap\,L^2\bigl([0,T];H^2(\Omega)\bigr)$;
\item[(b)] ~$\sqrt{\rho_\veps}\,u_\veps\,\rightharpoonup\,u$ in $L^\infty\bigl([0,T];L^2(\Omega)\bigr)$ and
$\sqrt{\rho_\veps}\,Du_\veps\,\rightharpoonup\,Du$ in $L^2\bigl([0,T];L^2(\Omega)\bigr)$;
\item[(c)] ~$r_\veps\,\ra\,r$ and $\rho_\veps^{3/2}\,u_\veps\,\ra\,u$ (strong convergence) in $L^2\bigl([0,T];L^2_{loc}(\Omega)\bigr)$,
\end{itemize}
where $r$ and $u$ are linked by the relation found in Proposition \ref{p:TP} above. Moreover,
$r$ solves (in the weak sense) the modified Quasi-Geostrophic equation
\begin{equation} \label{eq:q-geo_0}
\d_t\Bigl(\bigl(\Id-\Delta_h+\Delta_h^2\bigr)r\Bigr)\,+\,\nabla^\perp_h\bigl(\Id-\Delta_h\bigr)r\,\cdot\,\nabla_h\Delta_h^2r\,+\,
\frac{\nu}{2}\,\Delta_h^2\bigl(\Id-\Delta_h\bigr)r\,=\,0
\end{equation}
supplemented with the initial condition $r_{|t=0}\,=\,\wtilde{r}_0$, where $\wtilde{r}_0\,\in\,H^3(\R^2)$ is the unique solution of
$$
\bigl(\Id-\Delta_h+\Delta_h^2\bigr)\,\wtilde{r}_0\,=\,\int_0^1\bigl(\omega^3_0\,+\,r_0\bigr)\,dx^3\,,
$$
with $r_0$ and $u_0$ defined as the weak limits (up to extraction) of the initial data and $\omega_0\,=\,\nabla\times u_0$ the vorticity of $u_0$.
\end{thm}

%%%%%
\subsection{The result for variable rotation axis} \label{ss:var-result}

%Let us present here our main assumptions and results. We also recall basic facts on global existence of weak solutions to our model.
%\subsubsection{General hypotheses} \label{sss:hyp}
%\subsubsection{Existence of weak solutions} \label{sss:weak}
%\subsubsection{The singular perturbation problem} \label{sss:sing_results}

We get now interested in studying the incompressible and high rotation limit simultaneously, in the regime of constant capillarity and for effectively variable
rotation axis, i.e. when the function $\mf{c}$ is non-constant.

Here we will consider the general instance of \emph{ill-prepared} initial data
$\bigl(\rho,u\bigr)_{|t=0}=\bigl(\rho_{0,\veps},u_{0,\veps}\bigr)$. Namely, we will suppose
the following assumptions: % on the family $\bigl(\rho_{0,\veps}\,,\,u_{0,\veps}\bigr)_{\veps>0}$:
\begin{itemize}
\item[(i)] ~$\rho_{0,\veps}\,=\,1\,+\,\veps\,r_{0,\veps}$, with
$\bigl(r_{0,\veps}\bigr)_\veps\,\subset\,H^1(\Omega)\cap L^\infty(\Omega)$ bounded;
\item[(ii)] ~$1/\rho_{0,\veps}\,=\,1\,+\,\veps\,a_{0,\veps}$, with
$\bigl(a_{0,\veps}\bigr)_\veps\,\subset\,L^2(\Omega)$ bounded;
\item[(iii)] ~~$\bigl(u_{0,\veps}\bigr)_\veps\,\subset\,L^2(\Omega)$ bounded.
\end{itemize}
Up to extraction of a subsequence, we can suppose to have the weak convergence properties
\begin{equation} \label{eq:conv-initial}
r_{0,\veps}\,\rightharpoonup\,r_0\quad\mbox{ in }\;H^1(\Omega)\;,\qquad
a_{0,\veps}\,\rightharpoonup\,a_0\,=\,-\,r_0\;\mbox{ and }\;
u_{0,\veps}\,\rightharpoonup\,u_0\quad\mbox{ in }\;L^2(\Omega)\,.
\end{equation}

Remark that the previous assumptions respect what we required in Remark \ref{r:en_initial}, see Subsection \ref{ss:hypotheses}.
We also recall that it is at this point that we need condition (iv) of Definition \ref{d:weak}.

Let us turn our attention to the function $\mf c$.
For technical reason, analogously to what done in \cite{G-SR_2006}, we need to assume that it has non-degenerate
critical points: namely, we will suppose
\begin{equation} \label{eq:non-crit}
\lim_{\delta\ra0}\;\mc{L}\left(\left\{x^h\,\in\,\R^2\;\Bigl|\;\bigl|\nabla_h\mf{c}(x^h)\bigr|\,\leq\,\delta\right\}\right)\,=\,0\,,
\end{equation}
where we denoted by $\mc L(\mc O)$ the $2$-dimensional Lebesgue measure of a set $\mc O\,\subset\R^2$.

Also, fixed an admissible modulus of continuity $\mu$ (we will recall the precise definition in Paragraph \ref{sss:continuity}), we require that $\nabla_h\mf{c}$ belongs to the space
\begin{equation} \label{def:Z}
\mc{Z}_\mu(\R^d)\,:=\,\biggl\{a\,\in\,L^\infty(\R^d;\R)\;\biggl|\;\sup_{x\in\R^d}\left|a(x+y)+a(x-y)-2a(x)\right|\,\leq\,C\,\mu(|y|)\quad\forall\;|y|\leq1\biggr\}\,.
\end{equation}
Let us set $|a|_{\mc{Z}_\mu}$ as the smallest constant $C$ such that the previous inequality holds true, and
$\|a\|_{\mc{Z}_\mu}\,:=\,\|a\|_{L^\infty}+|a|_{\mc{Z}_\mu}$.

Finally,for notation convenience, we introduce the operator
$$
\mf{D}_{\mf{c}}(f)\,:=\,D_h\bigl(\mf{c}^{-1}\,\nabla_h^\perp f\bigr)\,=\,\frac{1}{2}\,\left(\nabla_h\,+\,^t\nabla_h\right)\bigl(\mf{c}^{-1}\,\nabla_h^\perp f\bigr)
$$
for any scalar function $f\,=\,f(x^h)$.

\begin{thm} \label{t:sing_var}
Let $1<\g\leq2$ in \eqref{eq:def_P} and $\mf{C}(\rho,u)=\mf{c}(x^h)\,e^3\times\rho\,u$,
where $\mf{c}\in W^{1,\infty}(\R^2)$ is $\neq0$ almost everywhere and it verifies the non-degeneracy condition \eqref{eq:non-crit}.
Let us also assume that $\nabla_h\mf{c}\,\in\,\mc{Z}_\mu$, for some admissible modulus of continuity $\mu$ which verifies the property
\begin{equation} \label{cond:mu}
\wtilde{\mu}(s)\,:=\,\mu(s)\,\log\!\left(1+\frac{1}{s}\right)\;\longrightarrow\;0\qquad\qquad\mbox{ for }\quad s\ra0\,.
\end{equation}

Let $\bigl(\rho_{0,\veps},u_{0,\veps}\bigr)_\veps$ be initial data satisfying the hypotheses ${\rm (i)-(ii)-(iii)}$ and
\eqref{eq:conv-initial}, and let $\bigl(\rho_\veps\,,\,u_\veps\bigr)_\veps$ be a family of corresponding weak solutions to system
\eqref{eq:NSK-sing}-\eqref{eq:bc} in $[0,T]\times\Omega$, in the sense of Definition \ref{d:weak}. Suppose that the symmetriy properties of Remark \ref{r:period-bc} are verified.
Define $r_\veps\,:=\,\veps^{-1}\left(\rho_\veps-1\right)$.

Then, up to the extraction of a subsequence, one has the following convergence properties: 
\begin{itemize}
\item[(a)] ~$r_\veps\,\rightharpoonup\,r$ in $L^\infty\bigl([0,T];H^1(\Omega)\bigr)\,\cap\,L^2\bigl([0,T];H^2(\Omega)\bigr)$,
\item[(b)] ~$\sqrt{\rho_\veps}\,u_\veps\,\rightharpoonup\,u$ in $L^\infty\bigl([0,T];L^2(\Omega)\bigr)$ and
 $\sqrt{\rho_\veps}\,Du_\veps\,\rightharpoonup\,Du$ in $L^2\bigl([0,T];L^2(\Omega)\bigr)$,
\end{itemize}
where $r$ and $u$ verify the relations established in Proposition \ref{p:TP}. Moreover,
$r$ solves (in the weak sense) the equation
\begin{equation} \label{eq:lim_var}
\d_t\left(r\,-\,\div_{\!h}\!\left(\frac{1}{\mf{c}^2}\,\nabla_h\bigl(\Id-\Delta_h\bigr)r\right)\right)\,+\,
%\nu\,\nabla_h^\perp\cdot\left(\frac{1}{\mf{c}}\,\nabla_h\cdot\mf{D}_{\mf{c}}\bigl((\Id-\Delta_h)r\bigr)\right)\,=\,0
\nu\,\;^t\mf{D}_{\mf{c}}\,\circ\,\mf{D}_{\mf{c}}\bigl((\Id-\Delta_h)r\bigr)\,=\,0
\end{equation}
supplemented with the initial condition $r_{|t=0}\,=\,\wtilde{r}_1$, where $\wtilde{r}_1$ is defined by
$$
\wtilde{r}_1\,-\,\div_{\!h}\!\left(\frac{1}{\mf{c}^2}\,\nabla_h\bigl(\Id-\Delta_h\bigr)\wtilde{r}_1\right)\,=\,
\int_0^1\Bigl({\rm curl}_h\bigl(\mf{c}^{-1}\,u^h_0\bigr)\,+\,r_0\Bigr)\,dx^3\,.
$$
\end{thm}

Notice that we have the identity
$$
\,^t\mf{D}_{\mf{c}}\,\circ\,\mf{D}_{\mf{c}}(f)\,=\,
\nabla_h^\perp\cdot\left(\frac{1}{\mf{c}}\,\nabla_h\cdot\mf{D}_{\mf{c}}(f)\right)\,,
%\,=\,\sum_{j,k}\left(\nabla_h^\perp\right)^k\left(\frac{1}{\mf{c}}\,\d_j\left(\mf{D}_{\mf{c}}(f)\right)_{jk}\right)\,,
$$
where we used the notations $\div f$ and $\nabla\cdot f$ in an equivalent way.
We also remark here that, for $\mf{c}\equiv1$, this operator reduces to $(1/2)\Delta_h^2f$, according to Theorem \ref{th:sing}.

\begin{rem} \label{r:var_axis}
The fact that the limit equation is linear is remarkable.
As already pointed out in \cite{G-SR_2006} and \cite{F-G-GV-N}, this corresponds to a sort of turbulent behaviour of the fluid, where all the scales are mixed and
one can identify just an average horizontal motion.

From the technical viewpoint, the motivation is that, in the case of variable rotation axis, the limit motion is much more constrained than for constant axis; correspondingly,
the kernel of the singular perturbation operator is smaller.
\end{rem}

%%%%%%%%%%%%%%%%%%%%%%%%%%%%%%%%%%%%%%%%%%%%%%%%%%%%%%%%%%%%%%%%%%%%%%%%%%%%%%%%%%%%%%%%%%%%%%%%%%%%%%%%%%%%%%%%%%%%%%%%%%%%%%%%%%%%%%%%%%
\section{Proof of the main result} \label{s:proof}

In the present section we prove our main result, namely Theorem \ref{t:sing_var}.
It relies on compensated compactness arguments, firstly introduced by Lions and Masmoudi \cite{L-M}-\cite{L-M_CRAS} in the context of incompressible limit,
and later adapted by Gallagher and Saint-Raymond \cite{G-SR_2006} to the case of rotating fluids (see also \cite{F-G-GV-N}).

First of all, we give some insights on Zygmund conditions and the regularity hypothesis on the rotation function $\mf{c}$. Then, we present how passing to the limit in the weak formulation
of the equations, and we finally derive the limit system.

%%%%%
\subsection{Moduli of continuity and Zygmund conditions} \label{ss:moduli}

The present subsection is devoted to present in detail the regularity class which the rotation coefficient belongs to.

First of all, we recall some basic notions and properties related to admissible moduli of continuity. Then, we switch to the analysis of Zygmund type conditions;
we conclude with a fundamental lemma, which allows us to prove Theorem \ref{t:sing_var}.

We are going to make a broad use of tools from Fourier Analysis, and especially Littlewood-Paley theory.
For the sake of conciseness, we do not present the details here, and we refer e.g. to Chapter 2 of \cite{B-C-D}. % and \cite{M-2008} (Chapters 4 and 5).

Furthermore, for simplicity of exposition  we will deal with the $\R^d$ case; however, everything can be adapted to the $d$-dimensional
torus $\mbb{T}^d$, and then also to the case of $\R^{d_1}\times\mbb{T}^{d_2}$.

%%%%%%%%%%%%%%%%%%%%%%%%%%%%%%%%%%%%%%%%%%%%%%%%%%%%%%%%%%%%%%%%%%%%%%%%%%%%%%%%%%%%%%%%%%%%%%%%%%%%
\subsubsection{Admissible moduli of continuity} \label{sss:continuity}
%%%%%%%%%%%%%%%%%%%%%%%%%%%%%%%%%%%%%%%%%%%%%%%%%%%%%%%%%%%%%%%%%%%%%%%%%%%%%%%%%%%%%%%%%%%%%%%%%%%

In this paragraph we recall some fundamental definitions and properties about general moduli of continuity.

First of all, let us introduce the \emph{Littlewood-Paley decomposition}, based on a non-homogeneous dyadic partition of unity in the
Phase Space.

We fix a smooth radial function $\chi$ supported in the ball $B(0,2)$,  equal to $1$ in a neighborhood of $B(0,1)$
and such that $r\mapsto\chi(r\,e)$ is nonincreasing over $\R_+$ for all unitary vectors $e\in\R^d$. Set
$\varphi\left(\xi\right)=\chi\left(\xi\right)-\chi\left(2\xi\right)$ and $\vphi_j(\xi):=\vphi(2^{-j}\xi)$ for all $j\geq0$.

The dyadic blocks $(\Delta_j)_{j\in\Z}$ are defined by\footnote{Throughout we agree  that  $f(D)$ stands for 
the pseudo-differential operator $u\mapsto\mc{F}^{-1}(f\,\mc{F}u)$.} 
$$
\Delta_j:=0\ \hbox{ if }\ j\leq-2,\quad\Delta_{-1}:=\chi(D)\qquad\mbox{ and }\qquad
\Delta_j:=\varphi(2^{-j}D)\; \mbox{ if }\;  j\geq0\,.
$$
We  also introduce the low frequency cut-off operators: for any $j\geq0$,
\begin{equation} \label{eq:low-freq}
S_ju\;:=\;\chi\bigl(2^{-j}D\bigr)u\;=\;\sum_{k\leq j-1}\Delta_{k}u\,.
\end{equation}

Let us now present a basic definition.
\begin{defin} \label{d:mod-cont}
A \emph{modulus of continuity} is a continuous non-decreasing function $\mu:[0,1]\,\longrightarrow\,\R_+$ such that $\mu(0)=0$.

It is said to be \emph{admissible} if the function $\Gamma_\mu$, defined by the relation
$$
\Gamma_\mu(s)\,:=\,s\,\mu(1/s)\,,
$$
is non-decreasing on $[1,+\infty[\,$ and it verifies, for some constant $C>0$ and any $s\geq1$,
$$
\int_s^{+\infty}\sigma^{-2}\,\Gamma_\mu(\sigma)\,d\sigma\,\leq\,C\,s^{-1}\,\Gamma_\mu(s)\,.
$$
\end{defin}

Given a modulus of continuity $\mu$, we can define the space $\mc{C}_\mu(\R^d)$ as the set of real-valued functions $a\in L^\infty(\R^d)$
such that
$$
|a|_{\mc{C}_\mu}\,:=\,\sup_{|y|\in\,]0,1]}\frac{\left|a(x+y)\,-\,a(x)\right|}{\mu(|y|)}\,<\,+\infty\,.
$$
We also define $\|a\|_{\mc{C}_\mu}\,:=\,\|a\|_{L^\infty}\,+\,|a|_{\mc{C}_\mu}$.

On the other hand, for an increasing $\Gamma$ on $[1,+\infty[\,$, we define the space $B_\Gamma(\R^d)$ as the set
of real-valued functions $a\in L^\infty(\R^d)$ such that
$$
|a|_{B_\Gamma}\,:=\,\sup_{j\geq0}\frac{\left\|\nabla S_ja\right\|_{L^\infty}}{\Gamma(2^j)}\,<\,+\infty\,,
$$
where $S_j$ is the low-frequency cut-off operator of a Littlewood-Paley decomposition, as introduced above.
We also set $\|a\|_{B_\Gamma}\,:=\,\|a\|_{L^\infty}\,+\,|a|_{B_\Gamma}$.

One has the following result (see Proposition 2.111 of \cite{B-C-D}).
\begin{prop} \label{p:cont-equiv}
Let $\mu$ be an admissible modulus of continuity. Then $\mc{C}_\mu(\R^d)\,=\,B_{\Gamma_\mu}(\R^d)$,
and the respective norms are equivalent. Moreover, for any $a\in\mc{C}_\mu(\R^d)$ one has
$$
\left\|\Delta_ja\right\|_{L^\infty}\,\leq\,C\,\mu(2^{-j})
$$
for all $j\geq-1$, where the constant $C$ just depend on $\|a\|_{\mc{C}_\mu}$.
\end{prop}

Now we want to present a commutator lemma, which is fundamental in the proof of our main result.
Let us recall first the classical result (see Lemma 2.97 of \cite{B-C-D}).
\begin{lemma} \label{l:commutator}
Let $\theta\in\mc{C}^1(\R^d)$ such that $\bigl(1+|\,\cdot\,|\bigr)\what{\theta}\,\in\,L^1$. There exists a constant $C$ such that,
for any Lipschitz function $\ell\in W^{1,\infty}(\R^d)$ and any $f\in L^p(\R^d)$ and for all $\lambda>0$, one has
$$
\left\|\bigl[\theta(\lambda^{-1}D),\ell\bigr]f\right\|_{L^p}\,\leq\,C\,\lambda^{-1}\,\left\|\nabla\ell\right\|_{L^\infty}\,\|f\|_{L^p}\,.
$$
\end{lemma}
Going along the lines of the proof, it is easy to see that the constant $C$ depends just on the $L^1$ norm
of the function $|x|\,k(x)$, where $k\,=\,\mc{F}_\xi^{-1}\theta$ denotes the inverse Fourier transform of $\theta$.

Easy modifications of the proof of Lemma \ref{l:commutator} give a variation of the previous lemma. For simplicity, we restrict our attention to the case of $\theta$ in the Schwartz
class $\mc{S}(\R^d)$.
\begin{lemma} \label{l:comm_L^p}
Let $\theta\in\mc{S}(\R^d)$ and $(p_1,p_2,q)\in[1,+\infty]^3$ such that $1/q\,=\,1+1/p_2-1/p_1$.
Then there exists a constant $C$ such that,
for any $f\in L^{p_1}(\R^d)$, any $\ell\in W^{1,\infty}(\R^d)$ and all $\lambda>0$,
$$
\left\|\bigl[\theta(\lambda^{-1}D),\ell\bigr]f\right\|_{L^{p_2}}\,\leq\,C\,
\lambda^{-1}\,\left\|\nabla\ell\right\|_{L^\infty}\,\|f\|_{L^{p_1}}\,.
$$
The constant $C$ just depends on the $L^q$ norm of the function $|x|\,k(x)$, where $k\,=\,\mc{F}_\xi^{-1}\theta$ as above.
\end{lemma}

Let us consider now less regular functions $\ell$. The next result is proved in \cite{F_2016}.
\begin{lemma} \label{l:comm-less}
Let $\theta\in\mc{C}^1(\R^d)$ be as in Lemma \ref{l:commutator}, and let $\mu$ be an admissible modulus of continuity. Then, there exists
a constant $C$ such that, for any function $\ell\in\mc{C}_{\mu}(\R^d)$ and any $f\in L^p(\R^d)$ and for all $\lambda>1$, one has
$$
\left\|\bigl[\theta(\lambda^{-1}D),\ell\bigr]f\right\|_{L^p}\,\leq\,C\,\mu(\lambda^{-1})\,\left|\ell\right|_{\mc{C}_\mu}\,\|f\|_{L^p}\,.
$$
The constant $C$ only depends on the $L^1$ norms of the functions $k(x)$ and $|x|\,k(x)$.
\end{lemma}

Obviously, an extension of the previous result, in the same spirit of Lemma \ref{l:comm_L^p}, holds true.

%%%%%
\subsubsection{Zygmund-type regularity conditions} \label{sss:Z_mu}
%%%%%

Let us now focus on second order conditions, namely Zygmund conditions. Given $\mu$ an admissible modulus of continuity,
we have defined the space $\mc{Z}_\mu$ in \eqref{def:Z}.

It is clear that $\mc{C}_\mu\,\subset\,\mc{Z}_\mu$.
Notice that, if $\mu(s)=s$ then one recovers the classical Zygmund space, while if $\mu(s)=s\,|\log s|$ then $\mc{Z}_\mu$
coincides with the space of log-Zygmund functions.

Zygmund and log-Zygmund conditions were introduced by Tarama \cite{Tar} in studying well-posedness of hyperbolic Cauchy problems with low regularity coefficients.
We refer to \cite{C-DS-F-M_Z-syst} and the references therein for further details and results in the same direction, and to \cite{F-Z}
for applications to control problems.

By use of Littlewood-Paley decomposition, we study here some properties of the space $\mc{Z}_\mu$.
The following analysis extends well-known facts about Zygmund and log-Zygmund classes (see e.g. Chapter 2 of \cite{B-C-D}, \cite{F-Z}).
%The respective proofs are equivalent, but, for the sake of completeness we will give most of the details.

First of all, we want to characterize $\mc{Z}_\mu$ as a special Besov-type class. Recalling Proposition \ref{p:cont-equiv},
this will imply (again) $\mc{C}_\mu\subset\mc{Z}_\mu$ in terms of dyadic blocks.
\begin{prop} \label{p:zyg}
The space $\mc{Z}_\mu$ coincides with the Besov-type class
$$
\mc{B}_\mu(\R^d)\,:=\,\left\{a\,\in\,L^\infty(\R^d;\R)\;\Bigl|\;\left\|\Delta_ja\right\|_{L^\infty}\,\leq\,C\,\mu(2^{-j})
\quad\forall\;j\geq-1\right\}\,.
$$
Moreover, the $\mc{Z}_\mu$ and $\mc{B}_\mu$ norms are equivalent, where we have defined
$$
\left\|a\right\|_{\mc{B}_\mu}\,:=\,\sup_{j\geq-1}\!\left(\frac{1}{\mu(2^{-j})}\,\|\Delta_ja\|_{L^\infty}\right)\,.
$$
\end{prop}

\begin{proof}
First of all, let us take an $a\in\mc{Z}_\mu(\R^d)$. By Bernstein inequality, we immediately have
$\|\Delta_{-1}a\|_{L^\infty}\leq\|a\|_{L^\infty}$. Now let us denote by $h$ the inverse Fourier transform of $\vphi$:
since $\vphi$ is even and $\int h\,=\,\vphi(0)\,=\,0$, for any $j\geq0$ we can write
\begin{eqnarray*}
\Delta_ja(x) & = & 2^{jd}\int_{\R^d}h(2^jy)\,a(x-y)\,dy\;=\;2^{jd-1}\int_{\R^d}h(2^jy)\,\bigl(a(x+y)+a(x-y)\bigr)\,dy \\
& = & 2^{jd-1}\int_{\R^d}h(2^jy)\,\bigl(a(x+y)+a(x-y)-2a(x)\bigr)\,dy\,.
\end{eqnarray*}
From this we deduce that
$$
\left\|\Delta_ja\right\|_{L^\infty}\,\leq\,C\,2^{jd}\,\int_{\R^d}|h|(2^j\,y)\,\mu(|y|)\,dy\,.
$$
Let us split the previous integral according to the space decomposition $\R^d\,=\,\left\{|y|\leq2^{-j}\right\}\,\cup\,
\left\{|y|\geq2^{-j}\right\}$. For the former term, since $\mu$ is increasing we have
$$
2^{jd}\,\int_{|y|\leq 2^{-j}}|h|(2^j\,y)\,\mu(|y|)\,dy\,\leq\,\mu(2^{-j})\,\|h\|_{L^1}\,.
$$
For the latter term, instead, we make the non-decreasing function $\Gamma_\mu$ appear, and we estimate
\begin{eqnarray*}
2^{jd}\,\int_{|y|\geq2^{-j}}|h|(2^{j}\,y)\,\mu(|y|)\,dz & = &
2^{jd}\,\int_{|y|\geq2^{-j}}|h|(2^j\,y)\,\Gamma_\mu(|y|^{-1})\,|y|\,dy \\
& \leq & C\,\Gamma_\mu(2^j)\,2^{-j}\,\left\|\;|\cdot|\;h(\,\cdot\,)\right\|_{L^1}\;\leq\;C\,\mu(2^{-j})\,.
\end{eqnarray*}
We have thus proved that $\mc{Z}_\mu\,\subset\,\mc{B}_\mu$.

Let now fix $a\in\mc{B}_\mu(\R^d)$. Take $x\in\R^d$ and $|y|\leq1$: for any $n\in\N$ we have the decomposition
\begin{eqnarray*}
a(x+y)+a(x-y)-2a(x) & = & \sum_{m<n}\bigl(\Delta_m a(x+y)+\Delta_m a(x-y)-2\Delta_m a(x)\bigr)\,+ \\
& & \qquad+\,\sum_{m\geq n}\bigl(\Delta_m a(x+y)+\Delta_m a(x-y)-2\Delta_m a(x)\bigr)\,.
\end{eqnarray*}
By Taylor formula up to second order, the former sum can be estimated by the quantity
\begin{eqnarray*}
C\,|y|^2\sum_{m<n}\left\|\nabla^2\Delta_m a\right\|_{L^\infty}  & \leq & C\,|y|^2\sum_{m<n}2^{2m}\,\mu(2^{-m}) \\
 & \leq & C\,|y|^2\,\sum_{m<n}2^{m}\,\Gamma_\mu(2^{m})\;\leq\;C\,|y|^2\,\Gamma_\mu(2^n)\,2^n\,.
\end{eqnarray*}
For the latter, instead, we use directly the property of the dyadic blocks, finding
$$
\sum_{m\geq n}\bigl|\Delta_m a(x+y)+\Delta_m a(x-y)-2\Delta_m a(x)\bigr|\,\leq\,4\sum_{m\geq n}\left\|\Delta_ja\right\|_{L^\infty}
\,\leq\,C\sum_{m\geq n}\Gamma_\mu(2^m)\,2^{-m}\,.
$$
Since $\mu$ is admissible, we have the bound
$$
\sum_{m\geq n}\Gamma_\mu(2^m)\,2^{-m}\,\leq\,C\int_{2^n}^{+\infty}\tau^{-2}\,\Gamma_\mu(\tau)\,d\tau\,\leq\,C\,\Gamma_\mu(2^n)\,2^{-n}\,,
$$
and this finally implies
$$
\bigl|a(x+y)+a(x-y)-2a(x)\bigr|\,\leq\,C\,\mu(2^{-n})\left(|y|^2\,2^{2n}\,+\,1\right)\,.
$$
Now, the choice $|y|\,2^n\,\sim\,1$ completes the proof of the inclusion $\mc{B}_\mu\,\subset\,\mc{Z}_\mu$,
and then of the whole proposition.
\end{proof}

On the other hand, Zygmund conditions imply a control on the first variation of the function, for which one loses a logarithmic factor.
\begin{prop} \label{p:zyg-first}
For any $a\,\in\,\mc{Z}_\mu(\R^d)$, there exists $C_a>0$ such that, for all $|y|\leq1$,
$$
\sup_{x\in\R^d}\bigl|a(x+y)-a(x)\bigr|\,\leq\,C_a\,\mu(|y|)\,\log\!\left(1\,+\,\frac{1}{|y|}\right)\,.
$$
The constant $C_a$ just depends on $\|a\|_{\mc{Z}_\mu}$.
\end{prop}

\begin{proof}
Analogously to the previous proof, let us write
$$
a(x+y)-a(x)\,=\,\sum_{m<n}\bigl(\Delta_m a(x+y)-\Delta_m a(x)\bigr)\,+\,
\sum_{m\geq n}\bigl(\Delta_m a(x+y)-\Delta_m a(x)\bigr)\,.
$$
As done above, we can estimate then
\begin{eqnarray*}
\bigl|a(x+y)-a(x)\bigr| & \leq & |y|\sum_{m<n}\left\|\nabla\Delta_m a\right\|_{L^\infty}\,+\,
2\sum_{m\geq n}\left\|\Delta_m a\right\|_{L^\infty} \\
& \leq & C\,|y|\sum_{m<n}\Gamma_\mu(2^m)\,+\,C\sum_{m\geq n}\Gamma_\mu(2^m)\,2^{-m} \\
& \leq & C\,\Gamma_\mu(2^n)\left(|y|\,n\,+\,2^{-n}\right)\;=\;C\,\mu(2^{-n})\left(2^n\,n\,|y|\,+\,1\right)\,.
\end{eqnarray*}
Again, the choice $n\,\sim\,\log_2\bigl(1/|y|\bigr)$ completes the proof of the proposition.
\end{proof}

Finally, let us present the analogue of Lemma \ref{l:comm-less} for second order regularity hypotheses. This result will be fundamental in proving Theorem \ref{t:sing_var}.

\begin{lemma} \label{l:comm-zyg}
Let $\theta\in\mc{C}^1(\R^d)$ such that $\bigl(1+|\,\cdot\,|\bigr)\what{\theta}\,\in\,L^1$, and let $\mu$ be an admissible modulus of
continuity for which \eqref{cond:mu} holds true. Then, there exist a constant $C$ and a $\lambda_0>0$ such that, for any function
$a\in\mc{Z}_{\mu}(\R^d)$ and any $f\in L^p(\R^d)$ and for all $\lambda\geq\lambda_0$, one has
$$
\left\|\bigl[\theta(\lambda^{-1}D),a\bigr]f\right\|_{L^p}\,\leq\,
C\,\mu(\lambda^{-1})\,\log(1+\lambda)\,\left\|a\right\|_{\mc{Z}_\mu}\,\|f\|_{L^p}\,.
$$
The constant $C$ only depends on the $L^1$ norms of the functions $k:=\mc{F}_\xi^{-1}\theta$ and $|\cdot|\,k$, while $\lambda_0$
just depends on $\mu$.
\end{lemma}

\begin{proof}
We start by writing the identity
$$
\bigl[\theta(\lambda^{-1}D),a\bigr]f\,=\,\lambda^d\int_\R^dk\bigl(\lambda(x-y)\bigr)\,f(y)\,\bigl(a(x)\,-\,a(y)\bigr)\,dy\,.
$$
Therefore, by use of Proposition \ref{p:zyg-first} above and Young inequality, we are reconducted to bound
$$
\lambda^d\,\left\|k(\lambda\,\cdot\,)\,\wtilde{\mu}(\,|\cdot|\,)\right\|_{L^1}\,=\,
\lambda^d\,\int_{\R^d}|k|(\lambda\,z)\,\wtilde{\mu}(|z|)\,dz\,,
$$
where $\wtilde{\mu}$ has been defined in \eqref{cond:mu}. We split the integral in the regions
$\left\{|z|\leq\lambda^{-1}\right\}$ and $\left\{|z|\geq\lambda^{-1}\right\}$.
For the former term, if $\lambda$ is big enough, by condition \eqref{cond:mu} we have
$$
\lambda^d\,\int_{|z|\leq \lambda^{-1}}|k|(\lambda\,z)\,\wtilde{\mu}(|z|)\,dz\,\leq\,\wtilde{\mu}(\lambda^{-1})\,\|k\|_{L^1}\,\leq\,
C\,\mu(\lambda^{-1})\,\log(1+\lambda)\,.
$$
For the latter term, instead, we have the estimate
\begin{eqnarray*}
& &  \lambda^d\,\int_{|z|\geq \lambda^{-1}}|k|(\lambda\,z)\,\wtilde{\mu}(|z|)\,dz\;=\;
\lambda^d\,\int_{|z|\geq \lambda^{-1}}|k|(\lambda\,z)\,\Gamma_\mu(|z|^{-1})\,|z|\,\log\left(1+\frac{1}{|z|}\right)\,dz \\
& & \qquad\qquad\qquad\qquad \leq C\,\Gamma_\mu(\lambda)\,\lambda^{-1}\,\log(1+\lambda)\,\left\|\;|\cdot|\,k\right\|_{L^1}\;\leq\;
C\,\mu(\lambda^{-1})\,\log(1+\lambda)\,.
\end{eqnarray*}

The lemma is hence proved.
\end{proof}

%%%%%
\subsection{Convergence by compensated compactness} \label{ss:proof}

In this section we start proving Theorem \ref{t:sing_var}. The argument is analogous to the one given in \cite{F_2016} for moduli of continuity. There is only one point where the Zygmund condition
comes into play, i.e. in Proposition \ref{p:regular} below, for which we have to use Lemma \ref{l:comm-zyg}.

Nonetheless, for reader's convenience, we will give here most of the details. Indeed, this allows us to describe propagation of waves.
As pointed out in Subsection \ref{ss:i_rot}, here we will study interactions of acoustic, Poincar\'e and Rossby waves: the first ones are due to the the compressible part
of $u$, the second ones to the vertical component of $u$ and the third ones to the variations of the axis.

\medbreak
By Proposition \ref{p:TP}, we have identified in \eqref{eq:sing-op} the singular perturbation operator
$\mc{A}$, which has variable coefficients. So, spectral analysis tools (employed in \cite{F_2015} for constant rotation axis) are out of use here. Hence,
in order to prove convergence in the weak formulation of our equations, we have to resort then to a compensated compactness argument.

Let us consider tests functions $\phi\,\in\,\mc{D}\bigl([0,T[\,\times\Omega\bigr)$ and
$\psi\,\in\,\mc{D}\bigl([0,T[\,\times\Omega;\R^3\bigr)$ such that the couple $(\phi,\psi)$ belongs to
${\rm Ker}\,{\mc A}$. Recall that, by Proposition \ref{p:TP}, they satisfy
$$
\div\psi\,=\,0\qquad\mbox{ and }\qquad \mf{c}(x^h)\,e^3\times\psi\,+\,\nabla\bigl(\Id-\Delta\bigr)\phi\,=\,0\,.
$$
In particular, $\psi=\bigl(\psi^h,0\bigr)$ and $\phi$ just depend on the horizontal variable $x^h\in\R^2$ and they are linked
by the relation $\mf{c}\,\psi^h\,=\,\nabla^\perp_h\bigl(\Id-\Delta_h\bigr)\phi$. Finally, % combining this property with the divergence-free condition for $\psi$,
we infer also that $\nabla^\perp_h\bigl(\Id-\Delta_h\bigr)\phi\cdot\nabla_h\mf{c}\,=\,0$.

First of all, we evaluate the momentum equation on such a $\psi$: taking into account the previous properties, we end up with
\begin{eqnarray}
& & \hspace{-0.7cm} \int_\Omega\rho_{0,\veps}\,u_{0,\veps}\cdot\psi(0)\,dx\;=\; 
\int^T_0\!\!\int_\Omega\biggl(-\rho_\veps\,u_\veps\cdot\d_t\psi\,-\,\rho_\veps\,u_\veps\otimes u_\veps:\nabla\psi\,+ \label{eq:sing_weak} \\
& & \qquad
+\,\nu\,\rho_\veps\,Du_\veps:\nabla\psi\,+\,\frac{1}{\veps^2}\,\Delta\rho_\veps\,\nabla\rho_\veps\cdot\psi\,+\,
\frac{\mf{c}(x^h)}{\veps}\,e^3\times\rho_\veps\,u_\veps\cdot\psi\biggr)dx\,dt. \nonumber
\end{eqnarray}
The $\d_t$ and viscosity terms do not present any difficulty in passing to the limit. On the other hand, the rotation
term can be handled by use of the weak form of the mass equation, tested on
$\wtilde{\phi}\,=\,\bigl(\Id-\Delta_h\bigr)\phi$: we get
\begin{eqnarray*}
\frac{1}{\veps}\int^T_0\!\!\int_\Omega\mf{c}(x^h)\,e^3\times\rho_\veps\,u_\veps\cdot\psi & = & 
-\,\frac{1}{\veps}\int^T_0\!\!\int_\Omega\mf{c}(x^h)\,\rho_\veps\,u^h_\veps\cdot\left(\psi^h\right)^\perp \\
%%\;=\;\frac{1}{\veps}\int^T_0\!\!\int_\Omega\rho_\veps\,u^h_\veps\cdot\nabla_h\wtilde{\phi} \\
%=\;\frac{1}{\veps}\int^T_0\!\!\int_\Omega\rho_\veps\,u^h_\veps\cdot\nabla_h\wtilde{\phi}
& = & 
-\int_\Omega r_{0,\veps}\,\wtilde{\phi}(0)\,-\,\int^T_0\!\!\int_\Omega r_\veps\,\d_t\wtilde{\phi}\,,
\end{eqnarray*}
which obviously converges in the limit $\veps\ra0$.

In order to deal with the transport and the capillarity terms, we want to use the structure of the system. Therefore, first of all we need
to introduce a regularization of our solutions.

\subsubsection{Regularization and description of the oscillations}

Let us set $V_\veps\,:=\,\rho_\veps\,u_\veps$. We can write system \eqref{eq:NSK-sing} in the form
\begin{equation} \label{eq:ac-w_c}
\begin{cases}
\veps\,\d_tr_\veps\,+\,\div\,V_\veps\,=\,0 \\[1ex]
\veps\,\d_tV_\veps\,+\,\Bigl(\mf{c}(x^h)\,e^3\times V_\veps\,+\,\nabla\bigl(\Id\,-\,\Delta\bigr)r_\veps\Bigr)\,=\,\veps\,f_\veps\,,
\end{cases}
\end{equation}
where we have defined $f_\veps$ by the formula
\begin{eqnarray}
f_\veps & := & -\,\div\left(\rho_\veps u_\veps\otimes u_\veps\right)\,+\,\nu\,\div\left(\rho_\veps Du_\veps\right)\,- 
\label{eq:f_veps} \\
& & \qquad -\,\frac{1}{\veps^2}\nabla\Bigl(\Pi(\rho_\veps)-\Pi(1)-\Pi'(1)\left(\rho_\veps-1\right)\Bigr)\,+\,
\frac{1}{\veps^2}\bigl(\rho_\veps-1\bigr)\,\nabla\Delta\rho_\veps\,. \nonumber
\end{eqnarray}
Equations \eqref{eq:ac-w_c} have to be read in the weak sense, of course. In particular, from writing
\begin{eqnarray*}
\langle f_\veps,\psi\rangle & := & \int_\Omega\biggl(\rho_\veps u_\veps\otimes u_\veps:\nabla\psi\,-\,
\nu\,\rho_\veps Du_\veps:\nabla\psi\,-\,\frac{1}{\veps^2}\,\Delta\rho_\veps\,\nabla\rho_\veps\cdot\psi\,- \\
& & \hspace{-0.5cm} -\,\frac{1}{\veps^2}\,(\rho_\veps-1)\,\Delta\rho_\veps\,\div\psi\,+\,
\frac{1}{\veps^2}\Bigl(\Pi(\rho_\veps)-\Pi(1)-\Pi'(1)\left(\rho_\veps-1\right)\Bigr)\div\psi\biggr)dx \\
& = & \int_\Omega\Bigl(f^1_\veps:\nabla\psi\,+\,f^2_\veps:\nabla\psi\,+\,f^3_\veps\cdot\psi\,+\,f^4_\veps\,\div\psi\,+\,
f^5_\veps\,\div\psi\Bigr)\,dx
\end{eqnarray*}
and by a systematic use of uniform bounds,
we can easily see that $\bigl(f^1_\veps\bigr)_\veps$ and $\bigl(f^5_\veps\bigr)_\veps$ are uniformly bounded in $L^\infty_T\bigl(L^1\bigr)$,
and so is $\bigl(f^2_\veps\bigr)_\veps$ in $L^2_T\bigl(L^2\bigr)$; finally, $\bigl(f^3_\veps\bigr)_\veps$ and $\bigl(f^4_\veps\bigr)_\veps$
are bounded in $L^2_T\bigl(L^1\bigr)$.

Therefore, we deduce that the family $\bigl(f_\veps\bigr)_\veps$ is uniformly bounded in the space
$L^2_T\bigl(H^{-1}(\Omega)\,+\,W^{-1,1}(\Omega)\bigr)$, and then in particular in $L^2_T\bigl(H^{-s}(\Omega)\bigr)$ for any
$s>5/2$.

\medbreak
Now, for any $M>0$, let us consider the low-frequency cut-off operator $S_M$ of a Littlewood-Paley decomposition,
as introduced in \eqref{eq:low-freq} above, and let us define
$$
r_{\veps,M}\,:=\,S_Mr_\veps\qquad\qquad\mbox{ and }\qquad\qquad
V_{\veps,M}\,:=\,S_MV_\veps\,.
$$
The following result hods true.
\begin{prop} \label{p:regular}
For any fixed time $T>0$ and compact set $K\subset\Omega$, the  following convergence properties hold, in the limit for $M\longrightarrow+\infty$:
\begin{equation} \label{reg:convergence}
\begin{cases}
\; \sup_{\veps>0}\left\|r_\veps\,-\,r_{\veps,M}\right\|_{L^\infty_T(H^s(K))\,\cap\,L^2_T(H^{1+s}(K))}\,\longrightarrow\,0
\qquad\qquad\forall\; s<1 \\[1ex]
\; \sup_{\veps>0}\left\|V_\veps\,-\,V_{\veps,M}\right\|_{L^2_T(H^{-s}(K))}\,\longrightarrow\,0
\qquad\qquad\forall\; s>0\,.
\end{cases}
\end{equation}

Moreover, for any $M>0$, the couple $\bigl(r_{\veps,M}\,,\,V_{\veps,M}\bigr)$ satisfies the approximate wave equations
\begin{equation} \label{reg:approx-w}
\begin{cases}
\veps\,\d_tr_{\veps,M}\,+\,\div\,V_{\veps,M}\,=\,0 \\[1ex]
\veps\,\d_tV_{\veps,M}\,+\,\Bigl(\mf{c}(x^h)\,e^3\times V_{\veps,M}\,+\,\nabla\bigl(\Id-\Delta\bigr)r_{\veps,M}\Bigr)\,=\,
\veps\,f_{\veps,M}\,+\,g_{\veps,M}\,,
\end{cases}
\end{equation}
where $\bigl(f_{\veps,M}\bigr)_{\veps}$ and $\bigl(g_{\veps,M}\bigr)_{\veps}$ are families of smooth functions satisfying
\begin{equation} \label{reg:source}
\begin{cases}
\; \sup_{\veps>0}\left\|f_{\veps,M}\right\|_{L^2_T(H^{s}(K))}\,\leq\,C(s,M)
\qquad\qquad\forall\; s\geq0 \\[1ex]
\; \sup_{\veps>0}\left\|g_{\veps,M}\right\|_{L^2_T(H^1(K))}\,\longrightarrow\,0
\qquad\qquad\mbox{ for }\quad M\ra+\infty\,,
\end{cases}
\end{equation}
where the constant $C(s,M)$ depends on the fixed values of $s\geq0$, $M>0$.
\end{prop}

\begin{proof}
Keeping in mind the characterization of $H^s$ spaces in terms of Littlewood-Paley decomposition (see Chapter 2 of \cite{B-C-D}),
properties \eqref{reg:convergence} are straightforward consequences of the uniform bounds established in Paragraphs \ref{sss:BD} and \ref{sss:bounds}.

Next, applying operator $S_M$ to \eqref{eq:ac-w_c} immediately gives us system \eqref{reg:approx-w},
where, denoting by $[\mc{P},\mc{Q}]$ the commutator between two operators $\mc P$ and $\mc Q$, we have set
$$
f_{\veps,M}\,:=\,S_Mf_\veps\qquad\qquad\mbox{ and }\qquad\qquad 
g_{\veps,M}\,:=\,\bigl[\mf{c}(x^h),S_M\bigr]\bigl(e^3\times V_\veps\bigr)\,.
$$
By these definitions and the uniform bounds on $\bigl(f_\veps\bigr)_\veps$, it is easy to verify the first property in \eqref{reg:source}.
As for the second one, we need to proceed carefully.

First of all, by uniform bounds and Lemma \ref{l:commutator} we get
$$
\sup_{\veps>0}\left\|g_{\veps,M}\right\|_{L^2_T(L^2)}\,\leq\,C\,2^{-M}\,.
$$
As for the gradient, for any $1\leq j\leq 3$ we can write
$$
\d_jg_{\veps,M}\,=\,\left[\mf{c}\,,\,S_M\right]\d_j\left(e^3\times V_\veps\right)\,+\,
\left[\d_j\mf{c}\,,\,S_M\right]\left(e^3\times V_\veps\right)\,.
$$
In order to control the former term, we use Lemma \ref{l:comm_L^p} with $p_2=q=2$ and $p_1=1$. Recalling that, by \eqref{sing-b:D_rho-u},
$\left(DV_\veps\right)_\veps\subset L^2_T(L^1_{loc})$, for any compact $K\subset\Omega$ we get
$$
\sup_{\veps>0}\left\|\left[\mf{c}\,,\,S_M\right]\d_j\left(e^3\times V_\veps\right)\right\|_{L^2_T(L^2(K))}\,\leq\,C\,2^{-M}\,.
$$
For the latter term, instead, Lemma \ref{l:comm-zyg} gives us
$$
\sup_{\veps>0}\left\|\left[\d_j\mf{c}\,,\,S_M\right]\left(e^3\times V_\veps\right)\right\|_{L^2_T(L^2)}\,\leq\,C\,\mu(2^{-M})\,\log\bigl(1+2^{M}\bigr)\,.
$$

In the end, choosing $\eta(M)=\max\left\{2^{-M},\mu(2^{-M})\,\log\bigl(1+2^{M}\bigr)\right\}$ (which goes to $0$ when $M\ra+\infty$), we get
$$
\sup_{\veps>0}\left\|g_{\veps,M}\right\|_{L^2_T(H^1_{loc})}\,\leq\,C\,\eta(M)
$$
for a suitable constant $C>0$, and this completes the proof of the proposition.
\end{proof}

We also have an important decomposition for the approximated velocity fields. We refer to \cite{F_2016} for the proof.
\begin{prop} \label{p:decomp}
The following decompositions hold true: 
$$
V_{\veps,M}\,=\,\mbb{V}_{\veps,M}\,+\,\veps\,\mc{V}_{\veps,M}\qquad\qquad\mbox{ and }\qquad\qquad
DV_{\veps,M}\,=\,\mbb{D}_{\veps,M}\,+\,\veps\,\mc{D}_{\veps,M}\,,
$$
where, for any compact set $K\subset\Omega$ and any $s\geq0$ one has
$$
\begin{cases}
\left\|\mbb{V}_{\veps,M}\right\|_{L^2_T\bigl(L^2(K)\cap L^3(K)\bigr)}\,+\,
\left\|\mbb{D}_{\veps,M}\right\|_{L^2_T\bigl(L^2(K)\bigr)}\,\leq\,C(K) \\[1ex]
\left\|\mc{V}_{\veps,M}\right\|_{L^2_T\bigl(H^s(K)\bigr)}\,+\,
\left\|\mc{D}_{\veps,M}\right\|_{L^2_T\bigl(H^s(K)\bigr)}\,\leq\,C(K,s,M)\,,
\end{cases}
$$
for suitable positive constants $C(K)$, $C(K,s,M)$ depending just on the quantities in the brackets.
\end{prop}

Before proceeding, let us introduce some useful notations. More precisely, we recall the following  decomposition: for a vector-field $X$, we write
\begin{equation} \label{dec:vert-av}
X(x)\,=\,\langle X\rangle(x^h)\,+\,\wtilde{X}(x)\,,\quad\qquad\mbox{ where }\quad
\langle X\rangle(x^h)\,:=\,\int_{\mbb{T}}X(x^h,x^3)\,dx^3\,.
\end{equation}
Notice that $\wtilde{X}$ has zero vertical average, and therefore we can write $\wtilde{X}(x)\,=\,\d_3\wtilde{Z}(x)$,
with $\wtilde{Z}$ having zero vertical average as well.
%$$
%\wtilde{X}(x)\,=\,\d_3\wtilde{Z}(x)\,,\qquad\qquad\mbox{ with }\qquad
%\int_{\mbb{T}}\wtilde{Z}(x^h,x^3)\,dx^3\,=\,0\,.
%$$
We also set $\wtilde{Z}\,=\,\mc{I}(\wtilde{X})\,=\,\d_3^{-1}\wtilde{X}$.

\subsubsection{The capillarity term}

First of all, let us deal with the surface tension term in \eqref{eq:sing_weak}. Notice that it can be rewritten as
$\;\int^T_0\int_\Omega\Delta r_\veps\,\nabla r_\veps\cdot\psi\;$, for any smooth test function $\psi$.

Thanks to the next lemma, we reconduct ourselves to study the convergence in the case of regular density functions.
\begin{lemma} \label{l:capill_approx}
For any $\psi\,\in\,\mc{D}\bigl([0,T[\,\times\Omega;\R^3\bigr)$, we have
$$
\lim_{M\ra+\infty}\,\limsup_{\veps\ra0}\,\left|\int^T_0\!\!\int_\Omega\Delta r_\veps\;\nabla r_\veps\,\cdot\,\psi\,dx\,dt\,-\,
\int^T_0\!\!\int_\Omega \Delta r_{\veps,M}\; \nabla r_{\veps,M}\,\cdot\,\psi\,dx\,dt\right|\,=\,0\,.
$$
\end{lemma}

Then, for any $\psi\,\in\,\mc{D}\bigl([0,T[\,\times\Omega;\R^3\bigr)\,\cap\,{\rm Ker}\,{\mc{A}}$
we have to consider the convergence of the term (pay attention to the signs)
\begin{eqnarray*}
\hspace{-0.3cm} \int^T_0\!\!\int_\Omega\Delta r_{\veps,M}\;\nabla r_{\veps,M}\cdot\psi\,dx\,dt & = &
-\,\int^T_0\!\!\int_\Omega\bigl(\Id-\Delta\bigr)r_{\veps,M}\;\nabla r_{\veps,M}\cdot\psi\,dx\,dt\,+ \\
& & \qquad\qquad\qquad +\,\int^T_0\!\!\int_\Omega r_{\veps,M}\;\nabla r_{\veps,M}\cdot\psi\,dx\,dt\,.
%& = &
%-\,\int^T_0\!\!\int_\Omega\bigl(\Id-\Delta\bigr)r_{\veps,M}\;\nabla r_{\veps,M}\cdot\psi\,dx\,dt\,+\,
%\frac{1}{2}\int^T_0\!\!\int_\Omega\nabla\left(r_{\veps,M}^2\right)\cdot\psi\,dx\,dt
\end{eqnarray*}
Notice that $r_{\veps,M}\,\nabla r_{\veps,M}\,=\,\nabla\left(r_{\veps,M}\right)^2/2$: therefore, since $\div\psi=0$,
by integration by parts we get that the latter item on the right-hand side is identically $0$.

Hence, in the end we have to deal only with the remainder
\begin{eqnarray}
-\,\int^T_0\!\!\int_\Omega\bigl(\Id-\Delta\bigr)r_{\veps,M}\;\nabla r_{\veps,M}\cdot\psi & = & 
-\,\int^T_0\!\!\int_\Omega\bigl(\Id-\Delta_h\bigr)\lan r_{\veps,M}\ran\;\nabla_h\lan r_{\veps,M}\ran\cdot\psi\,- \label{comp-cpt:dens} \\
& & \qquad\qquad
-\,\int^T_0\!\!\int_\Omega\lan\bigl(\Id-\Delta\bigr)\wtilde{r}_{\veps,M}\;\nabla\wtilde{r}_{\veps,M}\ran\cdot\psi\,, \nonumber
\end{eqnarray}
where, using the notations of \eqref{dec:vert-av}, $\wtilde{r}_{\veps,M}$ denotes the mean-free part of $r_{\veps,M}$.

\subsubsection{The convective term} %\label{sss:convect}

We now deal with the convective term. Once again, the first step is to reduce
the study to the case of smooth vector fields $V_{\veps,M}$.

\begin{lemma} \label{l:conv_approx}
For any $\psi\,\in\,\mc{D}\bigl([0,T[\,\times\Omega;\R^3\bigr)$, we have
$$
\lim_{M\ra+\infty}\,\limsup_{\veps\ra0}\,\left|\int^T_0\!\!\int_\Omega\rho_\veps u_\veps\otimes u_\veps:\nabla\psi\,dx\,dt\,-\,
\int^T_0\!\!\int_\Omega V_{\veps,M}\otimes V_{\veps,M}:\nabla\psi\,dx\,dt\right|\,=\,0\,.
$$
\end{lemma}

Next, recall  equation \eqref{eq:sing_weak}: paying attention once again to the right signs, by the previous lemma we have just to pass to
the limit in the term
\begin{eqnarray*}
& & \hspace{-0.5cm} -\,\int^T_0\!\!\int_\Omega V_{\veps,M}\otimes V_{\veps,M}:\nabla\psi\,=\, 
\int^T_0\!\!\int_\Omega \div\left(V_{\veps,M}\otimes V_{\veps,M}\right)\,\cdot\,\psi \\
& &  \qquad\qquad\quad =\,\int^T_0\!\!\int_\Omega \div_h\left(\lan V^h_{\veps,M}\ran\otimes\lan V^h_{\veps,M}\ran\right)\,\cdot\,\psi\,+\,
\int^T_0\!\!\int_\Omega \div_h\left(\lan \wtilde{V}^h_{\veps,M}\otimes\wtilde{V}^h_{\veps,M}\ran\right)\,\cdot\,\psi \\
&  & \qquad\qquad\quad =\,\int^T_0\!\!\int_\Omega\left(\mc{T}^1_{\veps,M}\,+\,\mc{T}^2_{\veps,M}\right)\,\cdot\,\psi\,.
\end{eqnarray*}

For notational convenience, from now on we will generically denote by $\mc{R}_{\veps,M}$ any remainder, i.e. any term satisfying the property
\begin{equation} \label{eq:remainder}
\lim_{M\ra+\infty}\,\limsup_{\veps\ra0}\,\left|\int^T_0\int_\Omega \mc{R}_{\veps,M}\,\cdot\,\psi\,dx\,dt\right|\,=\,0
\end{equation}
for all test functions $\psi\,\in\,\mc{D}\bigl([0,T[\,\times\Omega;\R^3\bigr)\,\cap\,{\rm  Ker}\,{\mc{A}}$.

We give a sketch of how dealing with the terms $\mc{T}^1_{\veps,M}$ and $\mc{T}^2_{\veps,M}$, referring to Subsection 4.3.3 of \cite{F_2016}
for the details.

\paragraph{\bf Handling $\mc{T}^1_{\veps,M}$}

Since we are dealing with  smooth functions, we can integrate by parts: we get
\begin{eqnarray*}
\mc{T}^1_{\veps,M} & = & \div_{\!h}\!\left(\lan V^h_{\veps,M}\ran\otimes\lan V^h_{\veps,M}\ran\right)\;=\;
\div_{\!h}\!\bigl(\lan V^h_{\veps,M}\ran\bigr)\;\lan V^h_{\veps,M}\ran\,+\,
\lan V^h_{\veps,M}\ran\cdot\nabla_h\left(\lan V^h_{\veps,M}\ran\right) \\
& = & \div_{\!h}\bigl(\lan V^h_{\veps,M}\ran\bigr)\;\lan V^h_{\veps,M}\ran\,+\,
\dfrac{1}{2}\,\nabla_h\left(\left|\lan V^h_{\veps,M}\ran\right|^2\right)\,+\,
{\rm curl}_h\lan V^h_{\veps,M}\ran\;\lan V^h_{\veps,M}\ran^\perp\,.
\end{eqnarray*}
Notice that the second term is a perfect gradient, and then it vanishes when tested against a function in the kernel of the singular perturbation operator.

For the first term, we take advantage of system \eqref{reg:approx-w}: averaging the first equation with respect to $x^3$
and multiplying it by $\lan V^h_{\veps,M}\ran$, we arrive at
$$
\div_h\bigl(\lan V^h_{\veps,M}\ran\bigr)\;\lan V^h_{\veps,M}\ran\;=\;-\,\veps\,\d_t\lan r_{\veps,M}\ran\,\lan V^h_{\veps,M}\ran\;=\;
\mc{R}_{\veps,M}\,+\,\veps\,\lan r_{\veps,M}\ran\,\d_t\lan V^h_{\veps,M}\ran\,,
%-\,\veps\,\d_t\bigl(\lan r_{\veps,M}\ran \,\lan V^h_{\veps,M}\ran\bigr)\,+\,\veps\,\lan r_{\veps,M}\ran\,\d_t\lan V^h_{\veps,M}\ran
$$
since $\veps\,\d_t\bigl(\lan r_{\veps,M}\ran \,\lan V^h_{\veps,M}\ran\bigr)$ is a remainder in the sense specified by relation
\eqref{eq:remainder}. We use now the horizontal part of \eqref{reg:approx-w} (again, after taking the vertical average),
multiplied by $\lan r_{\veps,M}\ran$: paying attention to  the signs, we get
$$%\begin{eqnarray*}
\veps\,\lan r_{\veps,M}\ran\,\d_t\lan V^h_{\veps,M}\ran %& = & -\,\mf{c}(x^h)\,\lan r_{\veps,M}\ran\,\lan V^h_{\veps,M}\ran^\perp\,-\,
%\lan r_{\veps_M}\ran\,\nabla_h\bigl(\Id-\Delta_h\bigr)\lan r_{\veps_M}\ran\,+\,\veps\,\lan f^h_{\veps,M}\ran\,+\,\lan g^h_{\veps,M}\ran \\
% & = &
\,=\,-\,\mf{c}(x^h)\,\lan r_{\veps,M}\ran\,\lan V^h_{\veps,M}\ran^\perp\,+\,
\bigl(\Id-\Delta_h\bigr)\lan r_{\veps_M}\ran\,\nabla_h\lan r_{\veps_M}\ran\,+\,\mc{R}_{\veps,M}\,,
%%\,-\,\nabla_h\biggl(\bigl(\Id-\Delta_h\bigr)\lan r_{\veps_M}\ran\,\lan r_{\veps_M}\ran\biggr)\,+ \\
%%& & \qquad\qquad\qquad\qquad +\,\bigl(\Id-\Delta_h\bigr)\lan r_{\veps_M}\ran\,\nabla_h\lan r_{\veps_M}\ran\,+\,\mc{R}_{\veps,M}\,,
$$%\end{eqnarray*}
where we used also the properties proved in Proposition \ref{p:regular} and we included in the remainder term also the perfect gradient.
Inserting this relation into the expression for $\mc{T}^1_{\veps,M}$, we find that this term equals
\begin{equation} \label{eq:T^1_a}
%\mc{T}^1_{\veps,M}\,=\,
\mc{X}_{\veps,M}\,\lan V^h_{\veps,M}\ran^\perp
\,+\,\bigl(\Id-\Delta_h\bigr)\lan r_{\veps_M}\ran\,\nabla_h\lan r_{\veps_M}\ran\,+\,\mc{R}_{\veps,M}\,,
\end{equation}
where, for notational convenience, we set $\mc{X}_{\veps,M}\,:=\,{\rm curl}_h\lan V^h_{\veps,M}\ran\,-\,\mf{c}(x^h)\,\lan r_{\veps,M}\ran$.

In order to deal with the first term in the right-hand side, the idea is to decompose $V^h_{\veps,M}$ in the orthonormal basis
(up to normalization) $\left\{\nabla_h\mf{c}\,,\,\nabla^\perp_h\mf{c}\right\}$. Of course, this can be done in the region when
$\nabla_h\mf{c}$ is far from $0$: therefore, we proceed carefully.
First of all, we notice that, after some manipulations, we can write
\begin{equation} \label{eq:curl-cr}
\veps\,\d_t\mc{X}_{\veps,M}\,=\veps\,{\rm curl}_h\lan f^h_{\veps,M}\ran\,+\,{\rm curl}_h\lan g^h_{\veps,M}\ran\,+\,\lan V^h_{\veps,M}\ran\,\cdot\,\nabla_h\mf{c}(x^h)\,.
\end{equation}
Notice that, thanks to Proposition \ref{p:regular}, there exists a function $\eta\geq0$, with $\eta(M)\longrightarrow0$ for $M\ra+\infty$,
such that, for any compact $K\,\subset\,\Omega$,
\begin{equation} \label{est:curl_rem}
%\sup_{\veps>0}\left\|\veps\,{\rm curl}_h\lan f^h_{\veps,M}\ran\,+\,{\rm curl}_h\lan g^h_{\veps,M}\ran\right\|_{L^2_T\bigl(L^2(K)\bigr)}
\sup_{\veps>0}\left\|{\rm curl}_h\lan g^h_{\veps,M}\ran\right\|_{L^2_T\bigl(L^2(K)\bigr)}
\,\leq\,\eta(M)\,.
\end{equation}
Then, fixed a $b\in\mc{C}^{\infty}_0(\R^2)$, with $0\leq b(x^h)\leq1$, such that $b\equiv1$ on $\left\{|x^h|\leq1\right\}$ and
$b\equiv0$ on $\left\{|x^h|\geq2\right\}$, we define
$$
b_M(x^h)\,:=\,b\left(\bigl(\eta(M)\bigr)^{\!-1/2}\,\nabla_h\mf{c}(x^h)\right)\,.
$$

Now we are ready to deal with the first term in the right-hand side of \eqref{eq:T^1_a}.
On the one hand, using the decomposition
and the bounds established in Proposition \ref{p:decomp}, we deduce that, for any compact $K\subset\Omega$,
\begin{eqnarray*}
\left\|b_M\,\mc{X}_{\veps,M}\,\lan V^h_{\veps,M}\ran^\perp\right\|_{L^1([0,T]\times K)} & \leq & 
\veps\,C(M)\,+\,C\,\left\|b_M\right\|_{L^6(K)} \\
& \leq &  \veps\,C(M)\,+\,C\,\left(\mc{L}\left\{x^h\in\R^2\;\bigl|\;\left|\nabla_h\mf{c}(x^h)\right|\,\leq\,2\,\sqrt{\eta(M)}\right\}\right)^{\!1/6}\,.
\end{eqnarray*}
Therefore, thanks to hypothesis \eqref{eq:non-crit}, we infer that this term is a remainder, in the sense specified by
relation \eqref{eq:remainder}.
On the other hand, for $\nabla_h\mf{c}$ far from $0$, we can write
$$%\begin{eqnarray*}
\left(1-b_M\right)\mc{X}_{\veps,M}\lan V^h_{\veps,M}\ran^\perp
%& = & \left(1-b_M\right)\,\mc{X}_{\veps,M}
%\left(\frac{\lan V^h_{\veps,M}\ran^\perp\cdot\nabla_h^\perp\mf{c}}{\left|\nabla_h\mf{c}\right|^2}\,\nabla_h^\perp\mf{c}\,+\,
%\frac{\lan V^h_{\veps,M}\ran^\perp\cdot\nabla_h\mf{c}}{\left|\nabla_h\mf{c}\right|^2}\,\nabla_h\mf{c}\right) \\
=\left(1-b_M\right)\mc{X}_{\veps,M}
\left(\frac{\lan V^h_{\veps,M}\ran\cdot\nabla_h\mf{c}}{\left|\nabla_h\mf{c}\right|^2}\nabla_h^\perp\mf{c}+
\frac{\lan V^h_{\veps,M}\ran^\perp\cdot\nabla_h\mf{c}}{\left|\nabla_h\mf{c}\right|^2}\nabla_h\mf{c}\right).
$$%\end{eqnarray*}
We observe that the latter term in the right-hand side is identically $0$ when tested against a $\psi\in{\rm Ker}\,\mc{A}$.
%(because $\div_{\!h}\psi^h\,=\,\div_{\!h}\left(\mf{c}\psi^h\right)\equiv0$, and then $\psi^h\cdot\nabla_h\mf{c}=0$).
For the former term, instead, we use the expression found in \eqref{eq:curl-cr}: after some manipulations we get
\begin{eqnarray*}
\left(1-b_M\right)\,\mc{X}_{\veps,M}\frac{\lan V^h_{\veps,M}\ran\cdot\nabla_h\mf{c}}{\left|\nabla_h\mf{c}\right|^2}\,
\nabla_h^\perp\mf{c} & = & \frac{\veps\,\left(1-b_M\right)\,\d_t\left|\mc{X}_{\veps,M}\right|^2}{2\,\left|\nabla_h\mf{c}\right|^2}\,\nabla_h^\perp\mf{c}\,- \\
& & -\,\frac{\left(1-b_M\right)\,\mc{X}_{\veps,M}}{\left|\nabla_h\mf{c}\right|^2}
\left(\veps\,{\rm curl}_h\lan f^h_{\veps,M}\ran\,+\,{\rm curl}_h\lan g^h_{\veps,M}\ran\right)\,\nabla_h^\perp\mf{c}\,,
\end{eqnarray*}
which is again a remainder $\mc{R}_{\veps,M}$, thanks to Proposition \ref{p:decomp} and property \eqref{est:curl_rem}.

In the end, putting all these facts together, we have proved that (paying attention again to the right signs)
\begin{equation} \label{eq:T^1_final}
\mc{T}^1_{\veps,M}\,=\,\bigl(\Id-\Delta_h\bigr)\lan r_{\veps_M}\ran\,\nabla_h\lan r_{\veps_M}\ran\,+\,\mc{R}_{\veps,M}\,.
\end{equation}

\paragraph{\bf Dealing with $\mc{T}^2_{\veps,M}$} 
Let us now consider the term $\mc{T}^2_{\veps,M}$: exactly as done above, we can write
$$
\mc{T}^2_{\veps,M}\,=\,\lan \div_{\!h}\bigl(\wtilde{V}^h_{\veps,M}\bigr)\;\wtilde{V}^h_{\veps,M}\ran\,+\,
\dfrac{1}{2}\,\lan\nabla_h\left|\wtilde{V}^h_{\veps,M}\right|^2\ran\,+\,
\lan {\rm curl}_h\wtilde{V}^h_{\veps,M}\;\left(\wtilde{V}^h_{\veps,M}\right)^\perp\ran\,.
$$

Let us focus on the last term for a while: with the notations introduced in \eqref{dec:vert-av}, we have
$$
\left({\rm curl}\wtilde{V}_{\veps,M}\right)^h\,=\,\d_3\wtilde{W}^h_{\veps,M}\qquad\mbox{ and }\qquad
\left({\rm curl}\wtilde{V}_{\veps,M}\right)^3\,=\,{\rm curl}_h\wtilde{V}^h_{\veps,M}\,=\,\wtilde{\omega}^3_{\veps,M}\,,
$$
where we have defined $\wtilde{W}^h_{\veps,M}\,:=\,\left(\wtilde{V}^h_{\veps,M}\right)^\perp\,-\,
\d_3^{-1}\nabla^\perp_h\wtilde{V}^3_{\veps,M}$.
For these quantities, from \eqref{reg:approx-w}, taking the mean-free part and the ${\rm curl}$ we deduce
\begin{equation} \label{eq:mean-free}
\begin{cases}
\veps\,\d_t\wtilde{W}^h_{\veps,M}\,-\,\mf{c}\,\wtilde{V}^h_{\veps,M}\,=\,\left(\d_3^{-1}\,
{\rm curl}\left(\veps\,\wtilde{f}_{\veps,M}\,+\,\wtilde{g}_{\veps,M}\right)\right)^h \\[1ex]
\veps\,\d_t\wtilde{\omega}^3_{\veps,M}\,+\,\div_{\!h}\bigl(\mf{c}\,\wtilde{V}^h_{\veps,M}\bigr)\,=\,{\rm curl}_h\left(\veps\,\wtilde{f}^h_{\veps,M}\,+\,
\wtilde{g}^h_{\veps,M}\right)
\end{cases}
\end{equation}
Making use of the relations above and of Propositions \ref{p:regular} and \ref{p:decomp}, we get
\begin{eqnarray*}
{\rm curl}_h\wtilde{V}^h_{\veps,M}\left(\wtilde{V}^h_{\veps,M}\right)^\perp & = &
\frac{\veps}{\mf c}\d_t\left(\wtilde{W}^h_{\veps,M}\right)^\perp\wtilde{\omega}^3_{\veps,M}-\frac{\wtilde{\omega}^3_{\veps,M}}{\mf c}
\left(\d_3^{-1}{\rm curl}\left(\veps\wtilde{f}_{\veps,M}+\wtilde{g}_{\veps,M}\right)\right)^{h,\perp} \\
& = & \frac{1}{\mf c}\,\left(\wtilde{W}^h_{\veps,M}\right)^\perp\,\div_{\!h}\bigl(\mf{c}\,\wtilde{V}^h_{\veps,M}\bigr)\,+\,\mc{R}_{\veps,M}\,.
\end{eqnarray*}

Hence, including also the gradient term into the remainders and making some esy manipulations, we arrive at the equality
%$$ %\begin{eqnarray*}
%\mc{T}^2_{\veps,M}\,=\,\lan \div_{\!h}\bigl(\wtilde{V}^h_{\veps,M}\bigr)\,
%\left(\wtilde{V}^h_{\veps,M}\,+\,\left(\wtilde{W}^h_{\veps,M}\right)^\perp\right)\ran\,+\,
%\lan \frac{1}{\mf c}\,\left(\wtilde{W}^h_{\veps,M}\right)^\perp\,\wtilde{V}^h_{\veps,M}\cdot\nabla_h\mf{c} \ran\,+\,\mc{R}_{\veps,M}\,,
%$$
%which can be finally rewritten in the following way:
\begin{eqnarray*}
\mc{T}^2_{\veps,M} & = & \lan \div\wtilde{V}_{\veps,M}\,\left(\wtilde{V}^h_{\veps,M}\,+\,\left(\wtilde{W}^h_{\veps,M}\right)^\perp\right)\ran\,- \\
& & \qquad -\,\lan \d_3\wtilde{V}^3_{\veps,M}\,\left(\wtilde{V}^h_{\veps,M}\,+\,\left(\wtilde{W}^h_{\veps,M}\right)^\perp\right)\ran\,+\,
\lan \frac{1}{\mf c}\,\left(\wtilde{W}^h_{\veps,M}\right)^\perp\,\wtilde{V}^h_{\veps,M}\cdot\nabla_h\mf{c} \ran\,+\,\mc{R}_{\veps,M}\,.
\end{eqnarray*}
The second term on the right-hand side is actually another remainder.
As for the first term, instead, we use the equation for the density in \eqref{reg:approx-w} to obtain
$$ %\begin{eqnarray*}
\div\wtilde{V}_{\veps,M}\left(\wtilde{V}^h_{\veps,M}+\left(\wtilde{W}^h_{\veps,M}\right)^\perp\right)\,=\, 
\mc{R}_{\veps,M}\,+\,\veps\,\wtilde{r}_{\veps,M}\;\d_t\!\left(\wtilde{V}^h_{\veps,M}+\left(\wtilde{W}^h_{\veps,M}\right)^\perp\right)\,.
$$ %\end{eqnarray*}
Now, by equations \eqref{eq:mean-free} and \eqref{reg:approx-w} again, it is easy to see that
$$
\veps\,\wtilde{r}_{\veps,M}\;\d_t\!\left(\wtilde{V}^h_{\veps,M}+\left(\wtilde{W}^h_{\veps,M}\right)^\perp\right)\,=\,
\mc{R}_{\veps,M}\,+\,\nabla\wtilde{r}_{\veps,M}\,\bigl(\Id-\Delta\bigr)\wtilde{r}_{\veps,M}\,,
$$
and therefore we find (with attention to the right sign)
$$
\mc{T}^2_{\veps,M}\,=\,\lan \nabla\wtilde{r}_{\veps,M}\,\bigl(\Id-\Delta\bigr)\wtilde{r}_{\veps,M}\ran\,+\,
\lan \frac{1}{\mf c}\,\left(\wtilde{W}^h_{\veps,M}\right)^\perp\,\wtilde{V}^h_{\veps,M}\cdot\nabla_h\mf{c} \ran\,+\,\mc{R}_{\veps,M}\,. 
$$

Now we have to deal with the second term in this last identity. Once again, we take advantage of the decomposition along the basis $\left\{\nabla_h\mf{c}\,,\,\nabla^\perp_h\mf{c}\right\}$.
For simplicity of exposition, we omit the cut-off away from the region $\{\nabla_h\mf c=0\}$ and we just give a sketch of the argument, since it
is analogous to what done above for $\mc{T}^1_{\veps,M}$.

First of all, we write
$$
\frac{\wtilde{V}^h_{\veps,M}\cdot\nabla_h\mf{c}}{\mf c}\,\left(\wtilde{W}^h_{\veps,M}\right)^\perp\,=\,\frac{\wtilde{V}^h_{\veps,M}\cdot\nabla_h\mf{c}}{\mf c}
\left(\wtilde{W}^h_{\veps,M}\cdot\nabla_h\mf{c}\,\frac{\nabla_h^\perp\mf{c}}{\left|\nabla_h\mf{c}\right|^2}\,+\,\left(\wtilde{W}^h_{\veps,M}\right)^\perp\cdot\nabla_h\mf{c}\,
\frac{\nabla_h\mf{c}}{\left|\nabla_h\mf{c}\right|^2}\right).
$$
As before, the last term in the right-hand side vanishes when tested against a smooth $\psi\in{\rm Ker}\,{\mc A}$.
Next, from the first equation in \eqref{eq:mean-free} we get
$$
\wtilde{V}^h_{\veps,M}\cdot\nabla_h\mf{c}\,=\,\frac{1}{\mf c}\left(\veps\,\d_t\wtilde{W}^h_{\veps,M}\,-\,\left(\d_3^{-1}\,
{\rm curl}\left(\veps\,\wtilde{f}_{\veps,M}\,+\,\wtilde{g}_{\veps,M}\right)\right)^h\right)\cdot\nabla_h\mf{c}\,.
$$
Therefore, we obtain that
\begin{eqnarray*}
\frac{1}{\mf c}\,\left(\wtilde{W}^h_{\veps,M}\right)^\perp\,\wtilde{V}^h_{\veps,M}\cdot\nabla_h\mf{c} & = & \frac{\veps}{2\,\mf c^2}\,
\d_t\left|\wtilde{W}^h_{\veps,M}\cdot\nabla_h\mf{c}\right|^2\,\frac{\nabla_h^\perp\mf{c}}{\left|\nabla_h\mf{c}\right|^2}\,- \\
& & \quad-\,\frac{1}{\mf c^2}\,\left(\d_3^{-1}\,{\rm curl}\left(\veps\,\wtilde{f}_{\veps,M}\,+\,\wtilde{g}_{\veps,M}\right)\right)^h\cdot\nabla_h\mf{c}\;
\frac{\nabla_h^\perp\mf{c}}{\left|\nabla_h\mf{c}\right|^2}\,,
\end{eqnarray*}
which is obviously a remainder in the sense of relation \eqref{eq:remainder}.

In the end, we have discovered that (paying attention to the right sign)
\begin{equation} \label{eq:T^2_final}
\mc{T}^2_{\veps,M}\,=\,\lan \nabla\wtilde{r}_{\veps,M}\,\bigl(\Id-\Delta\bigr)\wtilde{r}_{\veps,M}\ran\,+\,\mc{R}_{\veps,M}\,. 
\end{equation}

\subsection{The limit equation}

Let us sum up what we have just proved. In order to pass to the limit in equation \eqref{eq:sing_weak}, we needed to treat the non-linearities coming
from the capillarity term and the convection term.

Putting relations \eqref{comp-cpt:dens}, \eqref{eq:T^1_final} and \eqref{eq:T^2_final} all together, we finally discover that
$$
\int^T_0\!\!\int_\Omega\biggl(-\,\bigl(\Id-\Delta\bigr)r_{\veps,M}\;\nabla r_{\veps,M}\cdot\psi\,-\,V_{\veps,M}\otimes V_{\veps,M}:\nabla\psi\biggr)\,=\,
\int^T_0\!\!\int_\Omega\mc{R}_{\veps,M}\cdot\psi\,,
$$
which immediately implies, together with Lemmas \ref{l:capill_approx} and \ref{l:conv_approx}, that
$$
\lim_{M\ra+\infty}\,\lim_{\veps\ra0}
\int^T_0\!\!\int_\Omega\biggl(\frac{1}{\veps^2}\,\Delta\rho_\veps\,\nabla\rho_\veps\cdot\psi\,-\,\rho_\veps\,u_\veps\otimes u_\veps:\nabla\psi\biggr)dx\,dt\,=\,0\,.
$$

Then, thanks to the previous computations, we can pass to the limit in the weak formulation of our system: we obtain
$$ %\begin{eqnarray*}
\int^T_0\!\!\int_\Omega\biggl(-u\cdot\d_t\psi-r\d_t\bigl(\Id-\Delta_h\bigr)\phi+\nu Du:\nabla\psi\biggr)dxdt\,=\,
\int_\Omega\bigl(u_0\cdot\psi(0)+r_0\bigl(\Id-\Delta_h\bigr)\phi(0)\bigr)dx
$$ %\end{eqnarray*}
for any $(\phi,\psi)$ test functions belonging to the kernel of the singular perturbation operator ${\mc{A}}$.
Recall that, in particular, this implies the relation $\mf{c}\,\psi^h\,=\,\nabla_h^\perp\bigl(\Id-\Delta_h\bigr)\phi$.
Furthermore, also $(r,u)\,\in\,{\rm Ker}\,{\mc{A}}$: then we have the properties $\div u\equiv0$, $u=\bigl(u^h,0\bigr)$ and
$\mf{c}\,u^h\,=\,\nabla_h^\perp\bigl(\Id-\Delta_h\bigr)r$.

Setting $X(r)\,=\,\bigl(\Id-\Delta_h\bigr)r$ and $\wtilde{\phi}\,=\,\bigl(\Id-\Delta_h\bigr)\phi$, and
using that all the functions depend just on the horizontal variables, straightforward computations yield to
\begin{eqnarray*}
-\int^T_0\int_\Omega u\cdot\d_t\psi\,dx\,dt & = &
\int^T_0\int_{\R^2}\div_{\!h}\!\left(\frac{1}{\mf{c}^2}\,\nabla_hX(r)\right)\,\d_t\wtilde{\phi}\,dx^h\,dt \\
\nu\int^T_0\int_\Omega Du:\nabla\psi\,dx\,dt & = & 
\nu\int^T_0\int_{\R^2}\mf{D}_{\mf{c}}\bigl(X(r)\bigr)\,:\,\nabla_h\bigl(\mf{c}^{-1}\,\nabla^\perp_h\wtilde{\phi}\,\bigr)\,dx^h\,dt \\ 
 & = & \nu\int^T_0\int_{\R^2}\,^t\mf{D}_{\mf{c}}\,\circ\,\mf{D}_{\mf{c}}\bigl(X(r)\bigr)\,\wtilde{\phi}\,dx^h\,dt\,.
\end{eqnarray*}

Inserting these equalities into the previous relation finally completes the proof of Theorem \ref{t:sing_var}.

%%%%%%%%%%%%%%%%%%%%%%%%%%%%%%%%%%%%%%%%%%%%%%%%%%%%%%%%%%%%%%%%%%%%%%%%%%%%%%%%%%%%%%%%%%%%
%%%%%%%%%%%%%%%%%%%%%%%%%%%%%%%%%%%%%%%%%%%%%%%%%%%%%%%%%%%%%%%%%%%%%%%%%%%%%%%%%%%%%%%%%%%%

\end{document}